\documentclass[review,onefignum,onetabnum]{siamart190516}
\nolinenumbers
\ifpdf
\hypersetup{
  pdftitle={Tensor Deblurring and Denoising Using Total Variation},
  pdfauthor={F. Sanogo, C. Navasca, and S. Kindermann}
}
\fi

\usepackage{graphicx,hyperref}
\usepackage{tikz,tikz-3dplot}
\usepackage{algorithm}
\usepackage{algorithmic}
\usepackage {verbatim}
\usepackage{bm}
\usepackage{graphicx}
\usepackage{cleveref}
\usepackage{xcolor}
\usepackage{caption}
\usepackage{subcaption}
\usepackage{amsmath,amssymb,eucal,amsfonts,enumerate}
\usepackage{tikz-cd}
\usepackage{mathtools}
\usepackage{booktabs}

 \DeclareMathOperator*{\argmin}{argmin}
\newsiamthm{lem}{Lemma}
\newsiamthm{thm}{Theorem}%[section]
\newsiamthm{prop}{Proposition}
\newsiamthm{defn}{Definition}%[section]
\newcommand{\grad}{\nabla}

\headers{Tensor Deblurring and Denoising Using Total Variation}{F. Sanogo,
C. Navasca, and S. Kindermann}

\title{Tensor Deblurring and Denoising Using Total Variation\thanks{Submitted to the editors October xx, 2021.
\funding{This work was funded by National Science Foundation under Grant Nos. DMS-1439786 and MCB-2126374.}}}

\author{Fatoumata Sanogo\thanks{Department of Statistical and Data Sciences, Smith College, Northampton, MA 01063
  (\email{fsanogo@smith.edu)}.}
  \and Carmeliza Navasca\thanks{Department of  Mathematics, University of Alabama at Birmingham, Birmingham, AL 35294
  (\email{cnavasca@uab.edu)}.}
\and Stefan Kindermann\thanks{Industrial Mathematics Institute, Johannes Kepler Universitat Linz, Altenbergerstrasse 69, A-4040 Linz, Austria
  (\email{kindermann@indmath.uni-linz.ac.at)}.}
}

\usepackage{amsopn}
\DeclareMathOperator{\diag}{diag}

\begin{document}

\maketitle

% REQUIRED
\begin{abstract}
We consider denoising and deblurring problems for tensors.
While images can be discretized as matrices, 
the analogous procedure for  color images or videos leads  to 
a tensor formulation. 
We extend the classical ROF functional for 
variational denoising and deblurring to the tensor case by employing multi-dimensional total variation regularization. Furthermore, the resulting minimization problem is calculated by 
the FISTA method generalized to the tensor case. We provide some numerical experiments 
by applying the scheme to the denoising, 
the deblurring, and the recoloring of color images
as well as to the deblurring of videos. 
\end{abstract}

% REQUIRED
%\begin{AMS}
 % 68Q25, 68R10, 68U05
%\end{AMS}

\section{Introduction}

%Tensor computations have become prevalent in across many fields in mathematics \cite{2006math7647D,105555270146}, computer science \cite{2017arXiv171104887G,8408516,7931588}, engineering \cite{chadwick1995} and data science \cite{Hou2017TensorbasedRM,1010071142799521}. In particular, tensor methods are now ubiquitous in areas of numerical linear algebra \cite{BECKMANN2005294,Beylkin05algorithmsfor}, imaging science \cite{8319500,8451531} and applied algebraic geometry \cite{KoBaKe05}. New tensor based methods are currently being developed in scientific computing of complex problems \cite{BOELENS2020109744}. 

Digital image restoration is an important task in image processing since it can be applied to various areas of applied sciences such as medical and astronomical imaging, film restoration, and image and video coding. There  are  already various methods 
(most of them not tensor-based)  to recover signal/images from noise and blurry observations, 
for instance,  statistical-based approaches \cite{IMG, Demoment}, 
methods employing Fourier and/or wavelet transforms \cite{Nee, Kat}, or variational methods~\cite{Rudin, Chambolle}. 
Among the variational methods, 
total variation (TV) regularization is one of the most prominent examples.  TV regularization, first introduced by Rudin, Osher, and Fatemi (ROF) in 1992 \cite{ROF}, has become one of the most used techniques in image processing and computer vision because it is known to remove noises while preserving sharp edges and boundaries. It has since evolved from an image denoising method into a more general technique applied to various inverse problems such as deblurring, blind deconvolution, and inpainting. 
 The results in this article are an extension to tensors of 
 the ROF methodology and the techniques describe by Teboulle and Beck in \cite{TeboulleBeck}. Teboulle and Beck's  approach was to minimize the TV-functional  using gradient projection,  proximal mappings and Nestorov acceleration, applied to the dual functional of ROF that was proposed by Chambolle in \cite{Antonin, Cham}.
Here we extend these algorithms to multi-channel images and videos.
After discretization, these objects can be represented as 
tensors.  For simplicity, most formulas that  we present are for third-order tensors. However, keep in mind that this work 
can be extended to any n-th order  tensor with little effort.

We provide  some details  for the total variation functional
in Section~\ref{Sec2}, describe the minimization algorithms 
that we use in Section~\ref{sec:fista}, and 
then present our numerical approach for the 
denoising and deblurring for tensors using total variation in the subsequent sections. 
We will provide some numerical results to prove the effectiveness of our approach using colored images and videos.

\section{Definition and Preliminaries}\label{Sec2}
Recall that in image denoising, the preservation of certain visible structure such as edges is essential. To achieve that goal, various methods were proposed like wavelet-based methods, stochastic methods, and variational methods, in particular, total variation regularization.
%We are going to use both the isotropic and anisotropic TV formulations. 

The Rudin-Osher-Fatemi (ROF) model for denoising  to obtain
a clean image $u$ from a noisy 
one $f$ proposes to  
solve the following optimization problem:
\begin{equation}\label{rofmodel}
\min_{u\in BV(\Omega)} \| u\|_{TV(\Omega)} + \frac{\lambda}{2} \|f-u\|^2_{L^2} \,dx,
\end{equation}
where $BV(\Omega )$ is the set of functions with bounded variation over the domain $\Omega$, $TV(\Omega )$ is the total variation over the domain, and $ \lambda$ is a penalty parameter.
The functional in \cref{rofmodel} is defined in the space of functions of bounded variation, so it does not necessarily require image functions to be continuous and smooth. That is why the denoised image   allows for jumps or discontinuities, i.e., edges,  and the visual quality of the 
result is superior to linear filter methods.

The ROF formulation do have some weaknesses like that it is not strictly convex, it is susceptible to backward diffusion, and it favors piecewise constant solutions, which may cause false edges (staircases) in the output image; see \cite{Mulet}. Also, in a naive numerical approach 
to minimize \eqref{rofmodel}, one has to add some small regularization parameter to calculate 
the derivative of $\| u\|_{TV(\Omega)}$  \cite{Chan,Vogel,Vese}, which can result in some loss in accuracy.

Furthermore, although the ROF formulation is efficient in preserving edges of uniform and small curvature, it may result in excessively smoothening of small scale features having more pronounced curvature edges as shown in \cite{Rao}. 
As Chan et al.~explained in \cite{Tony,Mulet,Val}, the TV approach may sometimes result in loss of contrast and geometry in the output images, even when the observed image is noise-free.

However, given the strength of TV-based techniques, especially in edge preservation, various approach have been proposed to solve the above mentioned issues, for instance, the 
total variation penalty method by Vogel in \cite{Vogel,Avog} or 
the  adaptive total variation based regularization model by 
Strong and Chan~\cite{Strong}.
Recently, many researchers have employed algorithms that use the
dual formulation of the ROF model; see \cite{Gol,Carter,Antonin}.

In the original formulation,  the isotropic TV seminorm 
$ \| u\|_{TV(\Omega)}$ is used,  
which essentially is the $L^1(\Omega)$-norm of 
the Euclidean norm of the gradient of a function (extended to nondifferentiable functions by a weak formulation).

However, the isotropic TV can sometimes round the sharp edges of an image such as at corners points. 
To overcome this, a modification of the ROF model \cref{rofmodel},
the anisotropic TV, 
was pointed out by Esedoglu and Osher \cite{Ani}.
Anisotropic TV helps avoid blur at the corners, and the difference 
to the isotropic case is the use of the $\ell^1$-norm of the gradient
in place of the Euclidean norm. 
 
  For numerical computations, only discretized versions of 
 the ROF-functionals are needed. 
 More precisely, after discretization, the image $f$ and 
 the denoised image $u$ in~\eqref{rofmodel} can be regarded 
 as matrices in $\mathbb{R}^{m\times n}$. 
  Let us denote the isotropic  TV by $TV_I$ and the  anisotropic TV by $TV_{L_1}$. 
 The discrete version of the TV-seminorms read as follows: 
  \begin{equation}%\label{TViso}
\begin{split}
&TV_I= \sum_{i=1}^{m-1}\sum_{j=1}^{n-1}\sqrt{(x_{i,j}-x_{i+1,j})^2
+(x_{i,j}-x_{i,j+1})^2}\\
\end{split}
\end{equation}
The anisotropic Total Variation seminorm on the other hand is defined as 
\begin{equation}%\label{TVani}
\begin{split}
TV_{L_1}(x)=  & \sum_{i=1}^{m-1}\sum_{j=1}^{n}\mid x_{i+1,j}-x_{i,j}\mid
+\sum_{i=1}^{m}\sum_{j=1}^{n-1}\mid x_{i,j+1}-x_{i,j}\mid\, .
\end{split}
\end{equation}
Our approach  works for either isotropic or anisotropic TV.

   One of the objectives of the current article is to extend 
 the ROF-formulation from the image case to the tensor case. 
 That is, we consider inputs $f$ and denoised elements $u$ 
 as being 3rd-order tensor, i.e., elements in 
 $\mathbb{R}^{m\times n \times o}$ (or more general, $n$-th order tensors). 
 These can be seen as discretizations of 3-variate functions 
 (e.g., gray-scale images depending on time or color images).

%%%%%%%%%%%%%%%%%%%%%%

Let us briefly sketch our tensor notation. 
We denote a vector by a bold lower-case letter $\mathbf{a}$.
The bold upper-case letter ${A}$ represents a matrix,
and the symbol of a tensor is a calligraphic letter $\mathcal{A}$.
Throughout this paper, we focus primarily on third-order and fourth-order tensors,
$\mathcal{A}=(a_{ijk})\in\mathbb{R}^{I\times J\times K}$ and $\mathcal{A}=(a_{ijkl})\in\mathbb{R}^{I\times J\times K \times L}$, respectively,  of indices $1\leq i\leq I,1\leq j\leq J$ and $1\leq l\leq K$,
but all definitions and calculations are applicable to tensors of any order $n > 2$. In some calculations, tensor unfolding (a.k.a matricization) is required.

A third-order tensor $\mathcal{A}$ has column, row and tube fibers,
which are defined by fixing every index but one and
denoted by ${a}_{:jk}$, ${a}_{i:k}$ and ${a}_{ij:}$ respectively.
Correspondingly, we can obtain three kinds
of matricizations 
${A}_{(1)} \in \mathbb{R}^{I \times J \cdot K},{A}_{(2)} \in \mathbb{R}^{J \times I \cdot K}$, and ${A}_{(3)} \in \mathbb{R}^{K \times I \cdot J}$ 
%
%$\mathbf{A}_{(1)},\mathbf{A}_{(2)}$ and $\mathbf{A}_{(3)}$ of %matricizations
of $\mathcal{A}$ according to respectively arranging the column, row, and tube fibers to be columns of matrices.

% For these matrices, we denote by $\mathbf{A}_{(1)} \in \mathbb{R}^{I \times J \cdot K},\mathbf{A}_{(2)} \in \mathbb{R}^{J \times I \cdot K}$ and $\mathbf{A}_{(3)} \in \mathbb{R}^{K \times I \cdot J}$ 
% the matricization of $\mathcal{A} \in \mathcal{R}^{I \times J \times K}$
% along the respective dimensions. 

%{\textcolor{blue}{
%We need to add more on matricization used in this paper.?}}
%We define tensor contractions here. }}

The inner product of $\mathcal{S} \in \mathbb{R}^{I \times J \times K}$ and $\mathcal{T} \in \mathbb{R}^{I \times J \times K}$ is given by \[\langle \mathcal{S}, \mathcal{T} \rangle = \sum_{ijk} s_{ijk} t_{ijk}.\] The inner product for third-order tensors is easily generalized to $n$th-order tensors; i.e., $\langle \mathcal{S}, \mathcal{T} \rangle = \sum_{i_1 i_2 \cdots i_n} s_{i_1 i_2 \cdots i_n} t_{i_1 i_2 \cdots i_n}$.

The contracted product can be viewed as a partial inner product in multiple indices. For example, given a third-order tensor
$\mathcal{S} \in \mathbb{R}^{I \times J \times K}$ and a second-order tensor $A \in \mathbb{R}^{J \times K}$, we define the contracted product as
\[\langle \mathcal{S}, {A} \rangle_{2,3;1,2} = \sum_{jk} s_{ijk} a_{jk},\]
where the modes 2, 3 and 1, 2 indicate the indices that are being contracted. In general, given any two tensors, 
$\mathcal{S} \in \mathbb{R}^{I_1 \times I_2 \cdots \times I_M}$ and $\mathcal{A} \in \mathbb{R}^{J_1 \times J_2 \cdots \times J_N}$, then the contracted product in the modes, $\bar{i}_1, \bar{i}_2, \cdots \bar{i}_l$, where $\bar{i}_l \in [1, \cdots, M]$ and $\bar{j}_1, \bar{j}_2, \cdots \bar{j}_l$, where $\bar{j}_l \in [1, \cdots, N]$  is defined as
\[\langle \mathcal{S}, \mathcal{A} \rangle_{\bar{i}_1, \bar{i}_2 \cdots \bar{i}_l;\bar{j}_1, \bar{j}_2 \cdots \bar{j}_l} = \sum_{k_1, \cdots , k_l} s_{i_1 i_2 k_1 i_3 k_2 \cdots i_{M-2} k_{l-1}  k_l i_M} a_{k_1 j_2 k_2 \cdots  j_{M-2}  k_{l-1} k_l},\]
where $k_m$ represents the index corresponding to the contracting mode pair $(\bar{i}_m,\bar{j}_m)$ given that the $i_m$ and $j_m$ indices are not contracting modes.

%Similarly, the mode multiplication for $V \in \mathbb{R}^{I \times J}$ and $\mathcal{S} \in \mathbb{R}^{I \times K \times J}$ is
%\[ (V^T)_1 \mathcal{S}= \sum_i s_{ijk}v_{ij}.  \]
%In general, the mode multiplication is defined as
%\[ \mathcal{C}_{i_l} \mathcal{S}= \sum_i s_{i_1 i_2 \cdot i_l \cdot i_N}v_{j_1 j_2 \cdots i_l \cdot j_M}  \]
%for 
%for $C \in \mathbb{R}^{J_1 \times J_2 \cdots \times J_M}$ and $\mathcal{S} \in \mathbb{R}^{I_1 \times I_2 \cdots  \times J_M}$.
%be defined for tensors of different numbers of modes and contractions can be made along any two conforming modes. For example, with a third order tensor $\mathcal{S} \in \mathbb{R}^{I \times J \times K}$ and a second order tensor (which is a matrix) $\mathbf{A} \in \mathbb{R}^{n \times m}$ we have 
%\[\langle \mathcal{S}, \mathbf{A} \rangle_{2,3;1,2} = \sum_{jk} s_{ijk} a_{jk}.\] , where the subscripts 2, 3 and 1, 2 indicate the contracting modes of the two argument.

The Frobenius norm of a tensor is given by 
 \[\Vert \mathcal{S} \Vert_F =\langle \mathcal{S}, \mathcal{S} \rangle^{1/2},\]
where $\mathcal{S} \in \mathbb{R}^{I \times J \times K}$.

The element-wise multiplication and division of two tensors $\mathcal{X} \in \mathbb{R}^{I\times J\times K}$ and $\mathcal{Y} \in \mathbb{R}^{I\times J\times K}$ are defined as
\begin{equation*}
   \mathcal{X}\cdot*\mathcal{Y} = x_{ijk}y_{ijk}
\end{equation*}
%The element-wise division of two tensors $\mathcal{X} \in \mathbb{R}^{I\times J\times K}$ and $\mathcal{Y} \in \mathbb{R}^{I\times J\times K}$ is:
and
\begin{equation*}
   \mathcal{X} \cdot / \mathcal{Y} = \frac{x_{ijk}}{y_{ijk}},
\end{equation*}
respectively.

\section{Minimization algorithms: ISTA, FISTA, and MFISTA}\label{sec:fista}
In this section, we describe the algorithms that 
are used to minimize the TV-regularization functional. 
%For our numerical results, we also follows the dual approach  of  
%Chambolle combined with  FISTA  as in \cite{BeckFast}.

Consider the following optimization problem:
\begin{equation}\label{prob}
    \min \{G(x)\equiv f(x)+g(x)\},
\end{equation}
where $f$ and $g$ are functionals with certain properties that are 
specified below.

\subsection{Iterative Shrinkage/Thresholding Algorithm (ISTA)}\label{ISTASEC}
ISTA (Iterative Shrinkage/Thresholding algorithm) is an iterative 
method for minimizing problems of the form \eqref{prob} in the 
case when $g$ is convex but not necessarily differentiable 
and $f$ is differentiable with  Lipschitz continuous gradient. 
It involves a combination of 
a usual gradient step for  $f$ and a proximal 
operator with respect to $g$.

This algorithm can be traced back to the proximal forward-backward iterative scheme introduced by Bruck in 1975 \cite{Bruck} and Passty in 1979 \cite{Passty} within the general framework of splitting methods.
Since the optimization problem \cref{prob} involves a minimization of a sum of a smooth term and a nonsmooth term, 
the ISTA model splits the minimization of $f$ and $g$, and it
iteratively calculates a minimizing sequence by 
\begin{eqnarray}\label{EqIsta}
%\scalebox{0.75}[1]{$
{x}_{n+1}= \arg\min\limits_{{x}} \left\{ f({x}_n)+\langle{x}
-{x}_n,\nabla f({x}_n)\rangle
+\frac{L}{2}\|{x}-{x}_n\|^2+g({x}) \right\},
%$}
\end{eqnarray}
where $f$ and $g$ are the smooth and nonsmooth functions, respectively.

The iteration can be written as 
\[ {x}_{n+1} = p_L({x}_{n}), \qquad p_L({x}_{n}) 
=  (I +  \tfrac{1}{L}\partial g)^{-1}\left({x}_{n} - \tfrac{1}{L}\nabla f(\mathbf{x}_{n}\right). \]
Here, $\partial g$ is the subgradient of $g$. 
In \cref{prob} if $g(x) \equiv 0$,  then ISTA is reduced to a smooth gradient method.
ISTA has a worst-case convergence rate  of $O(1/k)$ \cite{BeckFast}. 
The algorithms is summarized in \cref{alg:ISTA}.
\begin{algorithm}[htb]
\renewcommand{\algorithmicrequire}{\textbf{Input:}}
\caption{ISTA}
\label{alg:ISTA}
\begin{algorithmic}[1]
\REQUIRE $L \coloneqq L(f)$, where L is a Lipschitz constant of $\nabla(f)$.\\
\textbf{Step 0.} Take $x_0 \in \mathbb{R}^n$\\
\textbf{Step n.} ($n\geq 1$) Compute
\begin{equation*}%\label{fista1}
x_n = p_L\left(x_{n-1}\right)
\end{equation*}
\end{algorithmic}
\end{algorithm}

\subsection{Fast Iterative Shrinkage/Thresholding Algorithm (FISTA)}\label{ssec:fista}

Nesterov in \cite{Nest} showed that  for  smooth optimzation problems, there exists a gradient method with an $O(1/k^2)$ convergence rate, which is a substantial improvement compared to the rate for, e.g., ISTA. 
The method of Nesterov (Nesterov's acceleration)  only requires  one gradient evaluation at each iteration and  involves an additional recombination step with the previous iterate, which is easy to compute. FISTA is an extention of  Nesterov's acceleration  \cite{Nest} to the problem in \cref{prob}; the method is described  in \cref{alg:FISTA}.

% Consider the following optimization problem:

% \begin{equation}\label{prob}
%     \min \{G(x)\equiv f(x)+g(x)\}
% \end{equation}

%ISTA(Iterative Shrinkage/Thresholding algorithm) is a combination of proximal algorithm (soft thresholding) and and landweber algrithm. ISTA has a worst-case complexity result of $O(1/k)$ \cite{BeckFast}. In this section we will introduce a new ISTA with an improved complexity result of $O(1/k^2)$ called FISTA \cite{BeckFast}. FISTA is an accelerated variant of ISTA \cite{nesterov}. 

%Nesterov in \cite{Nest} showed that for such smooth problem there exists %a gradient method with an $O(1/k^2)$ complexity result. This method does %not require more than one gradient evaluation at each iteration but just %an additional point that is chosen which is easy to compute. FISTA is an %extention of Nesterov method \cite{Nest} to the problem in \cref{prob}. %Considering the proximal algorithm definition given in \cref{Proxsec}. 

\begin{algorithm}[htb]
\renewcommand{\algorithmicrequire}{\textbf{Input:}}
\caption{FISTA}
\label{alg:FISTA}
\begin{algorithmic}[1]
\REQUIRE An upper bound $L \geq L(f)$ on the Lipschitz constant L(f) of $\nabla(f)$.\\
\textbf{Step 0.} Take $y_1=x_0 \in \mathbb{R}^n$, $t_1=1$\\
\textbf{Step n.} ($n\geq 1$) Compute
\begin{equation}\label{fista1}
x_n = p_L\left(y_n\right)
\end{equation}
\vspace*{-12pt}
\begin{equation}\label{fista2}
t_{n+1} = \frac{1+\sqrt{1+4t_{n}^2}}{2}
\end{equation}
\vspace*{-12pt}
\begin{equation}\label{fista3}
y_{n+1} = x_n + \left( \frac{t_n-1}{t_{n+1}} \right) \left( x_{n}-x_{n-1} \right)
\end{equation}
\end{algorithmic}
\end{algorithm}

The main difference between the FISTA and ISTA is that the 
operator  $p_L(\cdot)$ is not employed at the previous point $x_{k-1}$ but rather at the point $y_k$ that is a linear combination of the previous two iterates \{$x_{k-1}, x_{k-2}$\}. Each iterate of FISTA depends on the previous two iterates and not only on the last iterate as in ISTA.
Note that FISTA is as simple as ISTA. They both share the same computational demand; the computation of the remaining additional steps is computationally negligible. Convergence of the FISTA algorithm was shown in \cite{BeckFast};  we state the main convergence result for FISTA in \cref{fistathm}.
\begin{thm}[\cite{BeckFast}, Theorem 4.1]\label{fistathm}
Let $\{x_k\}$ be generated by FISTA. Then for any $k\geq 1$
\begin{equation*}
    F(x_k)-F(x^*) \leq \frac{2L\|x_0-x^*\|^2}{(k+1)^2}, \qquad \forall x^*\in X^*.
\end{equation*}
\end{thm}

% The main difference between the FISTA and ISTA is that the proximal operator
% operator  $prox_L(\cdot)$ is not employed on the previous point $x_{k-1}$, but rather at the point $y_k$ which is a linear combination of the previous two points \{$x_{k-1}, x_{k-2}$\}. Each iterate of FISTA depends on the previous two iterates and not only on the last iterate as in ISTA.
% %Note that FISTA is as simple as ISTA. They both share the same computational demand; the computation of the remaining additional steps is computationally negligible. 
% The discussion on the convergence of FISTA algorithm was shown in \cite{BeckFast}. The main convergence result for FISTA is given in theorem \ref{fistathm}.

% \begin{thm}(\cite{BeckFast}, theorem 4.1)\label{fistathm}
% Let $\{x_k\}$ be generated by FISTA. Then for any $k\geq 1$
% \begin{equation*}
%     F(x_k)-F(x^*) \leq \frac{2L||x_0-x^*||^2}{(k+1)^2}, \forall x^*\in X^*
% \end{equation*}
 
% \end{thm}

%\subsection{Monotone Iterative Shrinkage/Thresholding Algorithm (MFISTA)}\label{ssec:mista}
Contrary to ISTA,  
FISTA is not a monotone algorithm, i.e., the 
functional values are not necessarily monotonically decreasing. 
Although it is not required for convergence, 
monotonicity is a desirable property.
Especially when an inexact version of FISTA is used, 
e.g., when the proximal mapping is only available 
approximately, monotonicity makes  the 
alogorithms become more robust. 
This issue brought the monotone version of FISTA, called MFISTA, which is described in \cref{alg:MFISTA}. The convergence result for MFISTA remain the same as for FISTA:
\begin{thm}[\cite{TeboulleBeck}, Theorem 5.1]\label{mfistathm}
Let $\{x_k\}$ be generated by FISTA. Then for any $k\geq 1$
\begin{equation*}
    F(x_k)-F(x^*) \leq \frac{2L\|x_0-x^*\|^2}{(k+1)^2}, \qquad \forall x^*\in X^*.
\end{equation*}
 
\end{thm}

\begin{algorithm}[htb]
\renewcommand{\algorithmicrequire}{\textbf{Input:}}
\caption{MFISTA \cite{TeboulleBeck}}
\label{alg:MFISTA}
\begin{algorithmic}[1]
\REQUIRE $L \geq L(f)$, where L is an upper bound on the Lipschitz constant of $\nabla(f)$.\\
\textbf{Step 0.} Take $y_1=x_0 \in \mathbb{R}^n$, $t_1=1$\\
\textbf{Step n.} ($n\geq 1$) Compute
\begin{equation*}%\label{fista1}
x_n = p_L\left(y_n\right)
\end{equation*}
\vspace*{-12pt}
\begin{equation*}%\label{fista2}
t_{n+1} = \frac{1+\sqrt{1+4t_{n}^2}}{2}
\end{equation*}
\vspace*{-12pt}
\begin{equation*}
x_n = \argmin \{F(x): x=z_n,x_{n-1}\}
\end{equation*}
\vspace*{-12pt}
\begin{equation*}%\label{fista3}
y_{n+1} = x_n + \left( \frac{t_n}{t_{n+1}} \right)(z_n-x_n)+ \left( \frac{t_n-1}{t_{n+1}} \right) \left( x_{n}-x_{n-1} \right)
\end{equation*}
\end{algorithmic}
\end{algorithm}

Employing FISTA for the ROF problem is at first sight an obvious
task 
by setting $g(x) = TV(x)$ and $f(x) = \frac{\lambda}{2} \|x-f\|^2$.
However, this leads to the difficulty of calculating the 
proximal mapping $(I + \partial g)^{-1}$ 
for $g$, which is not at all easy to do. 
Instead, it was proposed in \cite{TeboulleBeck} to apply FISTA
to the dual ROF problem. The dual problem has been identified and proposed
by Chambolle  \cite{Antonin, Cham} and essentially reads as 
\[ \min_p \| f - \mbox{div}\, p\|^2: \quad \text{s.t.} \  \|p\|_{L^\infty} \leq 1.  \] 
Thus, we may use for $g(x)$ in the FISTA scheme, 
the indicator function of 
$\|p\|_{L^\infty} \leq 1$.
The corresponding proximal mapping can be calculated directly and turns out to be a simple projection  operator. 
The analogous discrete versions 
are obtained by discretizing $\mbox{div}$ and all involved norms. 
More details on the implementation are described in the 
next sections for the tensor case.

\section{Tensor Denoising}\label{DenoiseSec}

In this section we extend the classical ROF functional to 
the tensor case. Instead of images, we consider higher-dimensional 
objects of interest, such as videos or color images. 
Then these high-dimensional objects can be described mathematically as multivariate functions.
After discretization, these  objects can be
modelled as higher-order tensors.

In analogy with the ROF-function,  let us consider the following discretized TV-based denoising problem for tensors: 

\begin{eqnarray}\label{denoise}
\min_{\mathcal{T}\in C} \left \Vert \mathcal{T} - \mathcal{S} \right \Vert^2_F + 2\lambda TV(\mathcal{T}),
\end{eqnarray}
where $\mathcal{T}$ is the desired unknown color image or video to be recovered, $\mathcal{S}$ is the observation (the noisy data), C is a closed convex set subset of  $\mathbb{R}^{m\times n\times o}$, 
 which models various constraints, 
and $\lambda$ is the regularization parameter (fidelity parameter) that provides a tradeoff between fidelity to measurement and noise sensitivity. Also, TV is the discrete total variation norm which can be the isotropic or anisotropic TV
defined as follows:
In the isotropic case, 
 \begin{equation}%\label{TViso}
\begin{split}
&TV_I= \\
&= \sum_{i=1}^{m-1}\sum_{j=1}^{n-1}\sum_{k=1}^{o-1}\sqrt{(x_{i,j,k}-x_{i+1,j,k})^2+(x_{i,j,k}-x_{i,j+1,k})^2+(x_{i,j,k}-x_{i,j,k+1})^2}.
%& +\sum_{i=1}^{m-1}\mid x_{i,n,o}-x_{i+1,n,o}\mid+\sum_{j=1}^{n-1}\mid x_{m,j,o}-x_{m,j+1,o}\mid+\sum_{k=1}^{o-1}\mid x_{m,n,k}-x_{m,n,k+1}\mid
\end{split}
\end{equation}
The anisotropic Total Variation seminorm on the other hand is defined as
\begin{equation}\label{TVani}
\begin{split}
TV_{L_1}(x)=  & \sum_{i=1}^{m-1}\sum_{j=1}^{n}\sum_{k=1}^{o}\mid x_{i+1,j,k}-x_{i,j,k}\mid+\sum_{i=1}^{m}\sum_{k=1}^{o}\sum_{j=1}^{n-1}\mid x_{i,j+1,k}-x_{i,j,k}\mid\\
& +\sum_{i=1}^{m}\sum_{k=1}^{o-1}\sum_{j=1}^{n}\mid x_{i,j,k+1}-x_{i,j,k}\mid.
\end{split}
\end{equation}
The norm $\| \cdot \|_F$ in \eqref{denoise} is the Frobenius norm for 
tensors. It is defined as $\| T  \|_F= \sum_{i,j,k} |t_{ijk}|^2$.

We may employ the algorithms of the previous section to 
the tensor problem~\eqref{denoise}. As stated before, a 
direct application of the (F)ISTA algorithms require 
the calculation of the proximal mapping, which is a nontrivial
task, so we will take a more tractable approach by using the dual formulation of Chambolle~\cite{Antonin}, 
which we describe in more detail in the following. 
This method that we will present below works for both isotropic and anisotropic TV.

Let us define the set $\mathbb{P}$ as  the tensor triplets $(\mathcal{P}, \mathcal{Q}, \mathcal{R})$, where $\mathcal{P}\in \mathbb{R}^{(m-1)\times n\times o}$, $\mathcal{Q}\in \mathbb{R}^{(m)\times (n-1)\times o}$ and $\mathcal{R}\in \mathbb{R}^{(m)\times n\times (o-1)}$ that satisfy the following conditions in case of anisotropic TV:
\begin{alignat*}{2}
\mid p_{i,j,k}\mid &\leq 1,&\qquad &i=1,\ldots,m-1, j=1,\ldots,n, k=1,\ldots,o\\
\mid q_{i,j,k}\mid &\leq 1,& &i=1,\ldots,m, j=1,\ldots,n-1, k=1,\ldots,o\\
\mid r_{i,j,k}\mid &\leq 1,&  &i=1,\ldots,m, j=1,\ldots,n, k=1,\ldots,o-1.
\end{alignat*} 
 For isotropic TV on the other hand the set $\mathbb{P}$ must satisfy the following:
 \begin{alignat*}{2} 
p_{i,j,k}^2+q_{i,j,k}^2+r_{i,j,k}^2&\leq 1,& \qquad &i=1,\ldots,m-1, j=1,\ldots,n-1, k=1,\ldots,o-1\\
\mid p_{i,n,o}\mid &\leq 1,& &i=1,\ldots,m-1\\ 
\mid q_{m,j,o}\mid &\leq 1,& &j=1,\ldots,n-1\\
\mid r_{m,n,k}\mid &\leq 1,& &k=1,\ldots,o-1.
\end{alignat*} 
Furthermore, let 
$\mathcal{L}:\mathbb{R}^{(m-1)\times n\times o}\times \mathbb{R}^{m\times (n-1)\times o}\times \mathbb{R}^{m\times n\times (o-1)}\rightarrow \mathbb{R}^{m\times n\times o}$ be a linear operator 
(essentially the discrete $\text{div}$-operator)
defined by:
$$\mathcal{L}(\mathcal{P}, \mathcal{Q}, \mathcal{R})_{i,j,k}=p_{i,j,k}+q_{i,j,k}+r_{i,j,k}-p_{i-1,j,k}-q_{i,j-1,k}-r_{i,j,k-1},$$ 
where $i=1,\ldots,m, j=1,\ldots,n, k=1,\ldots,o$. 
Here we assume zero at the boundaries, i.e., $$p_{0,j,k}=p_{m,j,k}=q_{i,0,k}=q_{i,n,k}=r_{i,j,0}=r_{i,j,o}=0.$$
The transpose of $\mathcal{L}$ denoted as  $$\mathcal{L}^T:\mathbb{R}^{m\times n\times o}\rightarrow \mathbb{R}^{(m-1)\times n\times o}\times \mathbb{R}^{m\times (n-1)\times o}\times \mathbb{R}^{m\times n\times (o-1)}$$ is given by
$\mathcal{L}^T(\mathcal{T})=(\mathcal{P}, \mathcal{Q}, \mathcal{R}),$
where the tensors $\mathcal{P}$, $\mathcal{Q}$ and $\mathcal{R}$ are defined as:
 \begin{alignat*}{2}  p_{i,j,k}&=t_{i,j,k}-t_{i+1,j,k}& \qquad 
 &\text{for } i=1,\ldots,m-1, j=1,\ldots,n, k=1,\ldots,K,\\ q_{i,j,k}&=t_{i,j,k}-t_{i,j+1,k}&  &\text{for } i=1,\ldots,m, j=1,\ldots,n-1, k=1,\ldots,K\\ 
 r_{i,j,k}&=t_{i,j,k}-t_{i,j,k+1}&  &\text{for } i=1,\ldots,m, j=1,\ldots,n, k=1,\ldots,K-1.
 \end{alignat*}

Now let us state the following proposition which is an extension to the N-di\-men\-si\-o\-nal case of the analogous one  in \cite{Chambolle, TeboulleBeck}. For simplicity we restrict ourselves to third-order tensors but the generalization to any $n$th-order tensor is straightforward. 
\begin{prop}
Let a tensor triplet $(\mathcal{P}, \mathcal{Q}, \mathcal{R})$ be the optimal solution of the optimization problem
%{\textcolor{blue}{there are some typos in the proposition statement below. The gradient of d is defined in 4.6. It should be defined before this theorem.}}
\begin{align}\label{prop}
&\min_{(\mathcal{P},\mathcal{Q},\mathcal{R})} d(\mathcal{P},\mathcal{Q},\mathcal{R}) %-\parallel\mathcal{H}_C(\textbf{S}-\lambda\mathcal{L}(\mathcal{P},\mathcal{Q},\mathcal{R})\parallel_F^2 + \parallel\textbf{S}-\lambda\mathcal{L}(\mathcal{P},\mathcal{Q},\mathcal{R})_F^2\parallel, \\
%&  d(\mathcal{P},\mathcal{Q},\mathcal{R}):= -\parallel\mathcal{H}_C(\textbf{S}-\lambda\mathcal{L}(\mathcal{P},\mathcal{Q},\mathcal{R})\parallel_F^2 + \parallel\textbf{S}-\lambda\mathcal{L}(\mathcal{P},\mathcal{Q},\mathcal{R})\parallel_F^2,
\end{align}
where $ d(\mathcal{P},\mathcal{Q},\mathcal{R}):= -\parallel\mathcal{H}_C(\textbf{S}-\lambda\mathcal{L}(\mathcal{P},\mathcal{Q},\mathcal{R})\parallel_F^2 + \parallel\textbf{S}-\lambda\mathcal{L}(\mathcal{P},\mathcal{Q},\mathcal{R})\parallel_F^2$ and $\mathcal{H}_C(x)=x-P_C(x)$ for every $x\in\mathbb{R}^{m\times n\times o}$. Then the optimal solution of our TV-based denoising problem (\ref{denoise}) is given by $x=P_C(\mathcal{S}-\lambda\mathcal{L}(\mathcal{P},\mathcal{Q},\mathcal{R})),$ where $P_C $ is the orthogonal projection operator onto the convex set $C$.
 \end{prop}

\begin{proof}
Here we will derive an equivalent dual formulation to problem (\ref{denoise}) in the form of (\ref{prop}). 
Note that $|x|= \max_p\{ px: \mid p\mid\leq1\}$. Thus, we can rewrite the anisotropic total variation \cref{TVani} as $$TV_{L_1}(x)= \max_{(\mathcal{P},\mathcal{Q},\mathcal{R})\in\mathbb{P}} X(\mathcal{T}, \mathcal{P}, \mathcal{Q}, \mathcal{R}),$$ where 
\begin{equation*}
\begin{aligned}
    X(\mathcal{T}, \mathcal{P}, \mathcal{Q}, \mathcal{R})={} &\sum_{j=1}^n\sum_{k=1}^K\sum_{i=1}^{m-1}p_{i,j,k}(t_{i+1,j,k}-t_{i,j,k})\\
    &+\sum_{i=1}^m\sum_{k=1}^K\sum_{j=1}^{n-1}q_{i,j,k}(t_{i,j+1,k}-t_{i,j,k})\\
    & +\sum_{i=1}^m\sum_{j=1}^n\sum_{k=1}^{K-1}r_{i,j,k}(t_{i,j,k+1}-t_{i,j,k}).
    \end{aligned}
\end{equation*}
Recall that the inner product of two same-sized tensor is given by \[\langle \mathcal{S}, \mathcal{T} \rangle = \sum_{i_1\hdots i_N} s_{i_1 \hdots  i_N} t_{i_1 \hdots i_N},\] for two tensors 
$\mathcal{S}, \mathcal{T} \in \mathbb{R}^{I \times J \times K}$.
Hence, the anisotropic TV is equivalent to $$TV_{L_1}(x)= \max_{(\mathcal{P},\mathcal{Q},\mathcal{R}) \in \mathbb{P}} \langle \mathcal{T},\mathcal{L}(\mathcal{P}, \mathcal{Q}, \mathcal{R})\rangle.$$ 
Therefore, the denoising optimization model (\ref{denoise}) become a min-max optimization problem: $$\min_{\mathcal{T}\in C} \max_{(\mathcal{P},\mathcal{Q},\mathcal{R}) \in \mathbb{P}} \left \Vert \mathcal{T} - \mathcal{S} \right \Vert_F^2 + 2\lambda  \langle\mathcal{T},\mathcal{L}(\mathcal{P}, \mathcal{Q}, \mathcal{R})\rangle$$
The min-max theorem (see chapter IV in \cite{Eke}) then allows us to interchange the min and max to get the following:
\begin{eqnarray}\label{minmax}
\max_{(\mathcal{P},\mathcal{Q},\mathcal{R}) \in \mathbb{P}}\min_{\mathcal{T}\in C} \{\left \Vert \mathcal{T} - \mathcal{S} \right \Vert_F^2 + 2\lambda  \langle\mathcal{T},\mathcal{L}(\mathcal{P}, \mathcal{Q}, \mathcal{R})\rangle\}.
\end{eqnarray}
% Also, recall that the Frobenius norm of a tensor is given by $\mathcal{S} \in \mathbb{R}^{I \times J \times K}$
 %\[\Vert \mathcal{S} \Vert_F =\langle \mathcal{S}, \mathcal{S} \rangle^{1/2}.\]
 Therefore, using elementary properties of inner products  and the definition of the Frobenius norm, the min-max problem (\ref{minmax}) becomes:
$$ \max_{(\mathcal{P}, \mathcal{Q}, \mathcal{R}) \in \mathbb{P} }\min_{\mathcal{T}\in C} \{\left \Vert \mathcal{T} - (\mathcal{S}-\lambda \mathcal{L}(\mathcal{P}, \mathcal{Q}, \mathcal{R}) \right \Vert_F^2 - \parallel \mathcal{S}-\lambda \mathcal{L}(\mathcal{P}, \mathcal{Q}, \mathcal{R})\parallel_F^2 + \parallel \mathcal{S}\parallel_F^2\}$$
Then the solution to  the inner optimization problem is 
 $\mathcal{T}= P_C(\mathcal{S}-\lambda \mathcal{L}(\mathcal{P}, \mathcal{Q}, \mathcal{R}))$. Plugging back this solution to the dual problem and ignoring the constant terms we get 
 $$ \max_{(\mathcal{P}, \mathcal{Q}, \mathcal{R}) \in \mathbb{P}} \{\left \Vert P_C(\mathcal{S}-\lambda \mathcal{L}(\mathcal{P}, \mathcal{Q}, \mathcal{R})) - (\mathcal{S}-\lambda \mathcal{L}(\mathcal{P}, \mathcal{Q}, \mathcal{R}) \right \Vert_F^2 - \parallel \mathcal{S}-\lambda \mathcal{L}(\mathcal{P}, \mathcal{Q}, \mathcal{R})\parallel_F^2$$
The proposition then follows.
\end{proof}

Note that a similar proof applies to the isotropic Total Variation semi-norm case since  $\sqrt{x^2+y^2}=\max_{p_1,p_2}\{p_1x+p_2x:p_1^2+p_2^2\leq\ 1\}$.

The dual problem (\ref{prop}) in the proposition can be solved using gradient type methods since the objective function is continuously differentiable.
 The gradient is then given by 
 \begin{equation}\label{grad}
 \nabla d(\mathcal{P}, \mathcal{Q}, \mathcal{R})=-2\lambda\mathcal{L}^T P_C(\mathcal{S}-\lambda\mathcal{L}(\mathcal{P}, \mathcal{Q}, \mathcal{R})).
 \end{equation}
 Recall that our goal is to combine FISTA (see \cref{sec:fista}) and the gradient projection method discussed above to solve our denoising problem. From the fundamental property for a smooth function in the class $C^{1,1}$, FISTA require an upper bound for the Lipschitz constant \cite{TeboulleBeck} of the gradient objective function of (\ref{prop}).
 
\begin{lem}
Let $L(d)$ be the Lipschitz constant of the gradient of the objective function $d$  given in \eqref{prop}. Then $$L(d)\leq 24\lambda^2.$$
\end{lem}
\begin{proof}
Let us consider $(\mathcal{P}_i, \mathcal{Q}_i, \mathcal{R}_i)$ such that $\mathcal{P}_i\in \mathbb{R}^{(m-1)\times n\times o}, \mathcal{Q}_i\in \mathbb{R}^{(m)\times (n-1)\times o}$ and $\mathcal{R}_i\in \mathbb{R}^{(m)\times n\times (o-1)}$, for $i=1,2,3$.
Then using \cref{grad} and the non-expansivity of the orthogonal projection operator \cite{projection}, we have:
\begin{equation*} %\label{eq1}
\begin{split}
&\|\grad d(\mathcal{P}_1, \mathcal{Q}_1, \mathcal{R}_1)- \grad d(\mathcal{P}_1, \mathcal{Q}_1, \mathcal{R}_1) - \grad d(\mathcal{P}_1, \mathcal{Q}_1, \mathcal{R}_1)\|\\  &=\Big\|-2\lambda\mathcal{L}^T P_C\left(\mathcal{S}-\lambda \mathcal{L}(\mathcal{P}_1, \mathcal{Q}_1, \mathcal{R}_1)\right)\\
& \qquad \qquad+ 2\lambda\mathcal{L}^T P_C\left(\mathcal{S}
-\lambda \mathcal{L}(\mathcal{P}_2, \mathcal{Q}_2, \mathcal{R}_2)\right) + 2\lambda\mathcal{L}^T P_C(\mathcal{S}-\lambda \mathcal{L}(\mathcal{P}_3, \mathcal{Q}_3, \mathcal{R}_3))\Big\| 
\\
 & \leq 2\lambda \|\mathcal{L}^T\|\Big\| P_C(\mathcal{S}-\lambda \mathcal{L}(\mathcal{P}_3, \mathcal{Q}_3, \mathcal{R}_3))
 + P_C(\mathcal{S}-\lambda \mathcal{L}(\mathcal{P}_2, \mathcal{Q}_2, \mathcal{R}_2))\\
 & \qquad \qquad - P_C(\mathcal{S}-\lambda \mathcal{L}(\mathcal{P}_1, \mathcal{Q}_1, \mathcal{R}_1))\Big\|\\
 & \leq 2\lambda^2\|\mathcal{L}^T\| 
 \big\|\mathcal{L}(\mathcal{P}_2, \mathcal{Q}_2, \mathcal{R}_2)
 +\mathcal{L}(\mathcal{P}_3, \mathcal{Q}_3, \mathcal{R}_3)
 -\mathcal{L}(\mathcal{P}_1, \mathcal{Q}_1, \mathcal{R}_1)\big\|\\
 &\leq 2\lambda^2\|\mathcal{L}^T\|\big\|\mathcal{L}\|\|(\mathcal{P}_2, \mathcal{Q}_2, \mathcal{R}_2)+(\mathcal{P}_3, \mathcal{Q}_3, \mathcal{R}_3)-(\mathcal{P}_1, \mathcal{Q}_1, \mathcal{R}_1)\big\|\\
 &= 2\lambda^2\|\mathcal{L}^T\|^2\big\|(\mathcal{P}_2, \mathcal{Q}_2, \mathcal{R}_2)+(\mathcal{P}_3, \mathcal{Q}_3, \mathcal{R}_3)-(\mathcal{P}_1, \mathcal{Q}_1, \mathcal{R}_1)\big\|.
\end{split}
\end{equation*}

% \begin{equation*} %\label{eq1}
% \begin{split}
% &||\grad d(p_1,q_1,r_1)- \grad d(p_1,q_1,r_1) - \grad d(p_1,q_1,r_1)||\\  &=||-2\lambda\mathcal{L}^T P_C(\mathcal{S}-\lambda \mathcal{L}(\mathbf{p}_{1},\mathbf{q}_{1},\mathbf{r}_{1}))+ 2\lambda\mathcal{L}^T P_C(\mathcal{S}-\lambda \mathcal{L}(\mathbf{p}_{2},\mathbf{q}_{2},\mathbf{r}_{2})) + 2\lambda\mathcal{L}^T P_C(\mathcal{S}-\lambda \mathcal{L}(\mathbf{p}_{3},\mathbf{q}_{3},\mathbf{r}_{3}))|| \\
%  & \leq 2\lambda ||\mathcal{L}^T|||| P_C(\mathcal{S}-\lambda \mathcal{L}(\mathbf{p}_{3},\mathbf{q}_{3},\mathbf{r}_{3}))+ P_C(\mathcal{S}-\lambda \mathcal{L}(\mathbf{p}_{2},\mathbf{q}_{2},\mathbf{r}_{2}))- P_C(\mathcal{S}-\lambda \mathcal{L}(\mathbf{p}_{1},\mathbf{q}_{1},\mathbf{r}_{1}))\\
%  & \leq 2\lambda^2||\mathcal{L}^T||||\mathcal{L}(\mathbf{p}_{2},\mathbf{q}_{2},\mathbf{r}_{2})+\mathcal{L}(\mathbf{p}_{3},\mathbf{q}_{3},\mathbf{r}_{3})-\mathcal{L}(\mathbf{p}_{1},\mathbf{q}_{1},\mathbf{r}_{1})||\\
%  &\leq 2\lambda^2||\mathcal{L}^T||||\mathcal{L}||||(\mathbf{p}_{2},\mathbf{q}_{2},\mathbf{r}_{2})+(\mathbf{p}_{3},\mathbf{q}_{3},\mathbf{r}_{3})-(\mathbf{p}_{1},\mathbf{q}_{1},\mathbf{r}_{1})||\\
%  &= 2\lambda^2||\mathcal{L}^T||^2||(\mathbf{p}_{2},\mathbf{q}_{2},\mathbf{r}_{2})+(\mathbf{p}_{3},\mathbf{q}_{3},\mathbf{r}_{3})-(\mathbf{p}_{1},\mathbf{q}_{1},\mathbf{r}_{1})||
% \end{split}
% \end{equation*}
Also,
\begin{equation} \label{eq1}
\begin{split}
\|\mathcal{L}^T(\mathcal{T})\|^2
&=\sum_{i=1}^{m-1}\sum_{j=1}^n\sum_{k=1}^{K}(t_{i,j,k}-t_{i+1,j,k})^2+\sum_{i=1}^{m}\sum_{j=1}^{n-1}\sum_{k=1}^{K}(t_{i,j,k}-t_{i,j+1,k})^2\\
& \hspace{4.8cm} +\sum_{i=1}^{m}\sum_{j=1}^n\sum_{k=1}^{K-1}(t_{i,j,k}-t_{i,j,k+1})^2\\
&\leq 2\sum_{i=1}^{m-1}\sum_{j=1}^n\sum_{k=1}^{K}(t_{i,j,k}^2-t_{i+1,j,k}^2)+2\sum_{i=1}^{m}\sum_{j=1}^{n-1}\sum_{k=1}^{K}(t_{i,j,k}^2-t_{i,j+1,k}^2)\\
& \hspace{4.82cm} +2\sum_{i=1}^{m}\sum_{j=1}^n\sum_{k=1}^{K-1}(t_{i,j,k}^2-t_{i,j,k+1}^2)\\
&\leq 12\sum_{j=1}^n\sum_{k=1}^{K}t_{i,j,k}^2.
\end{split}
\end{equation}
Then, using equations (\ref{eq1}), we  have $\|\mathcal{L}^T\|\leq\sqrt{12}$ since $\|\mathcal{L}^T(\mathcal{T})\|^2\leq \sqrt{12}\|\mathcal{T}\|$.
Therefore, $L(d) \leq 2\lambda^2\|\mathcal{L}^T\|^2\leq 24\lambda^2$.
\end{proof}

% Also,
% \begin{equation} \label{eq1}
% \begin{split}
% &||\mathcal{L}^T(\mathcal{T})||^2\\
% &=\sum_{i=1}^{m-1}\sum_{j=1}^n\sum_{k=1}^{K}(t_{i,j,k}-t_{i+1,j,k})^2+\sum_{i=1}^{m}\sum_{j=1}^{n-1}\sum_{k=1}^{K}(t_{i,j,k}-t_{i,j+1,k})^2+\sum_{i=1}^{m}\sum_{j=1}^n\sum_{k=1}^{K-1}(t_{i,j,k}-t_{i,j,k+1})^2\\
% &\leq 2\sum_{i=1}^{m-1}\sum_{j=1}^n\sum_{k=1}^{K}(t_{i,j,k}^2-t_{i+1,j,k}^2)+2\sum_{i=1}^{m}\sum_{j=1}^{n-1}\sum_{k=1}^{K}(t_{i,j,k}^2-t_{i,j+1,k}^2)+2\sum_{i=1}^{m}\sum_{j=1}^n\sum_{k=1}^{K-1}(t_{i,j,k}^2-t_{i,j,k+1}^2)\\
% &\leq 12\sum_{j=1}^n\sum_{k=1}^{K}t_{i,j,k}^2
% \end{split}
% \end{equation}
% Then, using equations (\ref{eq1}) $||\mathcal{L}^T||\leq\sqrt{12}$, since $||\mathcal{L}^T(\mathcal{T})||^2\leq \sqrt{12}||\mathcal{T}||$.

% Therefore, $L(d) \leq 2\lambda^2||\mathcal{L}^T||^2\leq 24\lambda^2$
% \end{proof}
%This works is summarize in \cref{algodenoise}.

Combining the results, we arrive at the algorithm that  is summarized in Algorithm~\ref{algodenoise}.
\begin{algorithm}
\renewcommand{\algorithmicrequire}{\textbf{Input:}}
\renewcommand\algorithmicensure {\textbf{Output:}}
\caption{Gradient Projection with FISTA.}
%\label{alg:Framwork}
\label{algodenoise}
\begin{algorithmic}[1]
\REQUIRE 

\STATE $\mathcal{S}$--observed image
\STATE $\lambda$--regularization parameter
\STATE N--Number of iterations.

\ENSURE $x^*= P_{C}(\mathcal{S}-\lambda \mathcal{L}(p,q,r))$
\STATE Give initial tensors:$(p_0,q_0,r_0)=(0_{(m-1)\times n \times o},0_{m\times (n-1) \times o},0_{m\times n\times (o-1)})$ and $t_1=1$
\STATE Update step for k=1,2,...,N \\
Compute
\begin{align*}&(\mathcal{P}_k,\mathcal{Q}_k,\mathcal{R}_k)=P_{\mathbb{P}}[(\mathcal{P}_{k-1},\mathcal{Q}_{k-1},\mathcal{R}_{k-1})+\frac{1}{12\lambda}\mathcal{L}^T(P_{C}[\mathcal{S} -\lambda\mathcal{L}(\mathcal{P}_{k-1},\mathcal{Q}_{k-1},\mathcal{R}_{k-1})])],\\
&t_{k+1}=\frac{1+\sqrt{1+4t_k^2}}{2},\\
&(\mathcal{P}_{k+1},\mathcal{Q}_{k+1},\mathcal{R}_{k+1})=(\mathcal{P}_{k},\mathcal{Q}_{k},\mathcal{R}_{k})+(\frac{t_k-1}{t_{k+1}})(\mathcal{P}_k-\mathcal{P}_{k-1},\mathcal{Q}_k-\mathcal{Q}_{k-1},\mathcal{R}_k-\mathcal{R}_{k-1}).
\end{align*}
\STATE Set $x^*= P_{C}(\mathcal{S}-\lambda \mathcal{L}(\mathcal{P}_{N},\mathcal{Q}_{N},\mathcal{R}_{N}))$.
\end{algorithmic}
\end{algorithm}

\subsection{Numerical experiments}
We provide some numerical results for tensor denoising problems.
All experiments were run using the MATLAB R2016B version.
For the first experiment, we implemented our algorithm for a $384\times384\times3$ colored image from the image processing toolbox of MATLAB. We added some random noise to the image. The recovered images are given with different Peak-to-Noise-Ratio (PSNR) in Db and $\lambda$ values. Our experiment is illustrated in \cref{fig:pepper}.  The stopping criteria for this case are assumed to be $m_{\text{max}} = 200$ and $\epsilon_{\text{tol}} = 10^{-6}$. This experiment shows the importance of choosing a good regularization parameter; $\lambda =100$ is an example of a bad regularization.

\begin{figure}
     \centering
     \begin{subfigure}[b]{0.3\textwidth}
         \centering
         \includegraphics[width=\textwidth]{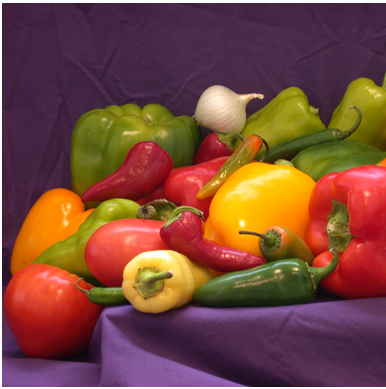}
         \caption{True Image}
         %\label{fig:y equals x}
     \end{subfigure}
     \hfill
     \begin{subfigure}[b]{0.3\textwidth}
         \centering
         \includegraphics[width=\textwidth]{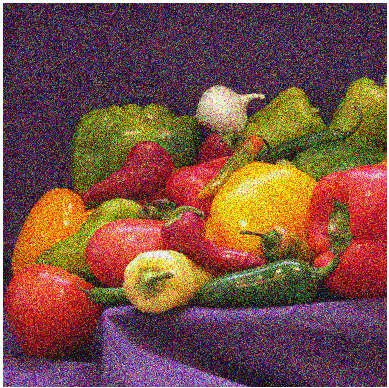}
         \caption{Input (PSNR 15.20)}
         %\label{fig:three sin x}
     \end{subfigure}
     \hfill
     \begin{subfigure}[b]{0.3\textwidth}
         \centering
         \includegraphics[width=\textwidth]{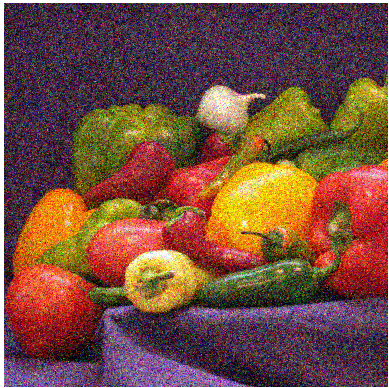}
         \caption{$\lambda=5$ (PSNR 17.39)}
         %\label{fig:five over x}
     \end{subfigure}
     
     \bigskip
\begin{subfigure}{0.3\linewidth}
  \centering
  \includegraphics[width=\textwidth]{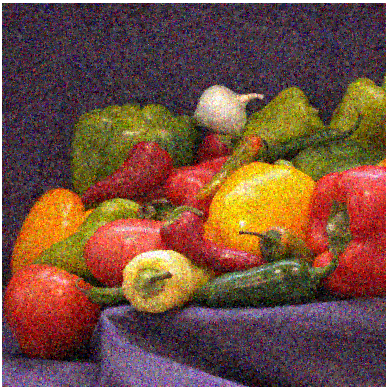}
  \caption{$\lambda=10$ (PSNR 19.68)}
  %\label{fig:image3}
\end{subfigure} 
\hfill
\begin{subfigure}{0.3\linewidth}
  \centering
  \includegraphics[width=\textwidth]{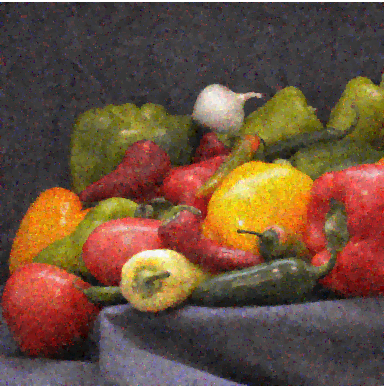}
  \caption{$\lambda=20$ (PSNR 22.47)}
  %\label{fig:image3}
\end{subfigure}
\hfill
\begin{subfigure}{0.3\linewidth}
  \centering
  \includegraphics[width=\textwidth]{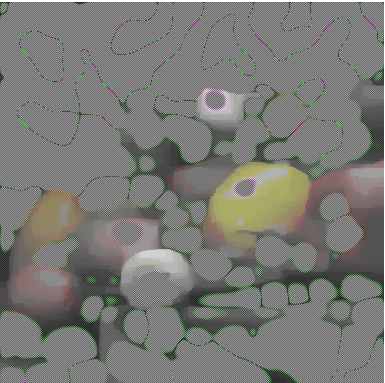}
  \caption{$\lambda=100$ (PSNR 6.19)}
  %\label{fig:image3}
\end{subfigure}
        \caption{TV-regularized denoising with increasing values of $\lambda$}
        \label{fig:pepper}
\end{figure}

For our next experiment we implemented our algorithm for a $246\times 246\times 3$ colored image, and we added some random noise to the image
as before.   Our experiment is illustrated in \cref{fig:brick}.
The recovered images are given with different PSNR in Db and $\lambda$ values. The stopping criteria for this case are assumed to be $m_{\text{max}} = 100$ and $\epsilon_{\text{tol}} = 10^{-6}$. We compared our results to the  Split Bregman method using the same stopping criteria. Our algorithm performed slightly better proven by the PSNR values.
Split Bregman \cite{Bregman} is a flexible algorithm for solving non-differentiable convex minimization problems, and it is especially efficient for problems with $L_1$- or TV-regularization.

\begin{figure}
     \centering
     \begin{subfigure}[b]{0.4\textwidth}
         \centering
         \includegraphics[width=.8\linewidth]{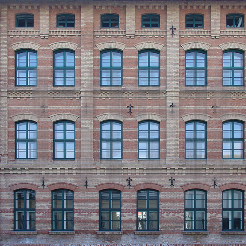}
         \caption{True Image}
        % \label{fig:y equals x}
     \end{subfigure}
     %\hfill
     \begin{subfigure}[b]{0.4\textwidth}
         \centering
         \includegraphics[width=.8\linewidth]{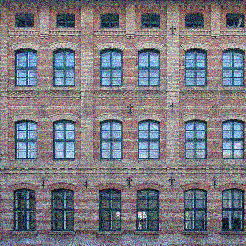}
         \caption{Input (PSNR 18.71)}
         %\label{fig:three sin x}
     \end{subfigure}
     
     \bigskip
\begin{subfigure}[b]{0.4\textwidth}
         \centering
         \includegraphics[width=.8\linewidth]{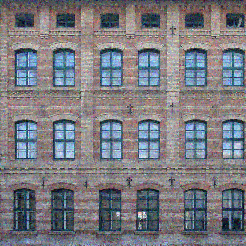}
         \caption{Our method (PSNR 22.35)}
         %\label{fig:five over x}
     \end{subfigure}
    \begin{subfigure}[b]{0.4\textwidth}
  \centering
  \includegraphics[width=.8\linewidth]{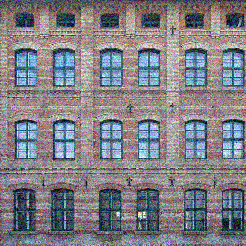}
  \caption{Split Bregman (PSNR 19.61)}
  %\label{fig:image3}
\end{subfigure}

%\hfill
%\begin{subfigure}{0.4\linewidth}
  %\centering
  %\includegraphics[width=\textwidth]{brick_denoise1.png}
  %\caption{$\lambda=10$ (PSNR 24.11)}
  %\label{fig:image3}
%\end{subfigure} 
\caption{TV-regularized denoising with increasing values of $\lambda$.}
\label{fig:brick}
\end{figure}

\section{Tensor Deblurring}
In this section, we present a second class of 
TV-re\-gu\-la\-ri\-za\-tion for tensors, namely the extension 
of the classical deblurring problem. 

Blurring always occurs in images or video recordings sometimes due to the fact that the optical system in a camera lens may be out of focus such that the incoming light is smeared out. Blurring can also occur in astronomical imaging when the light in a telescope is bent by turbulence.
The goal in image deblurring is to recover the original image by using a mathematical model of the blurring process. The issue here is the lost of some hidden information present in the blurred image. Unfortunately, there is no hope of recovering the exact image due to unavoidable errors like fluctuations in the recording process and approximation errors when representing the image with a limited number of digits. In image deblurring the biggest challenge is to create an algorithm to recover as much information as possible from the given data.

The ROF-model can be  extended without difficulties to the deblurring  problem by replacing the fidelity term $\|f-u\|^2$
by $\|f-K u\|^2$, where $K$ is a model for a blurring 
operator (e.g., a convolution operator).

Analogously, we may extend the (discretized) 
TV-deblurring functional 
to the tensor case by considering 
 the following deblurring functional:
\begin{eqnarray}\label{deblur}
\min_{\mathcal{T}\in C} \left \Vert \mathcal{A}(\mathcal{T}) - \mathcal{S} \right \Vert^2_F + 2\lambda TV(\mathcal{T}),
\end{eqnarray}
where $\mathcal{T}$ in $\mathbb{R}^{m\times n\times o}$ is the desired unknown image to be recovered, $\mathcal{A}$:$\mathbb{R}^{m\times n\times o}\to \mathbb{R}^{m\times n\times o}$ is a linear transformation representing some blurring operator, $\mathcal{S}$ is the observed noisy and blurred data, C is a closed convex set subset of  $\mathbb{R}^{m\times n\times o}$, and $\lambda$ is the regularization parameter. TV is the discrete total variation semi-norm as above,  either  isotropic ($TV_I$) or anisotropic ($TV_{L_1}$). 
Looking at problem (\ref{deblur}), we notice that the deblurring problem is a little more challenging than the denoising problem because of the 
additional blurring operator $\mathcal{A}$.

If we construct an equivalent smooth optimization algorithm for \eqref{deblur} via the approach in \cref{DenoiseSec}, i.e., by FISTA,  we will need to invert the linear transformation matrix $\mathcal{A}$ in the proximal operator, 
which cannot be done by a simple threshold operator as for the denoising case. 

To avoid this difficulty, we treat the TV deblurring problem \eqref{deblur} in two steps through the denoising problem;  see, e.g.,  \cite{TeboulleBeck}.
More precisely, the method iteratively solves in each step 
a deblurring problem with $f$ replaced by $x_n  - 
\frac{2}{L} A^*(A x_n-f)$. Thus, this approach  requires at each iteration the solution of the denoising problem \eqref{denoise} in addition to the gradient step.

\subsection{Details of the implementation}
At first we focus on  the blurring of  3D images. 
Convolving a point spread function (PSF) with a true image gives  a blurred image. In our case, we assume the PSF known and given. 
The PSF is the function that describes the blurring and the resulting image of the point source. 

% The first thing we will focus on is how to blur an image. For that we will be using the concept of convolution in 3D image. Convolving a Point Spread Function (PSF) with a true image gives you a blurred image. The PSF will be experimentally generated.  The PSF is the function that describes the blurring and the resulting image of the point source. Mathematically, the point source can be for instance equivalent to defining an array of all zeros, except a single pixel whose value is 1.
Here we define convolution in N-D as follows:
\begin{align*} 
&(f * g)((n_1,n_2,...,n_n) \\
& \quad := \sum_{j_1=-\infty}^\infty\sum_{j_2=-\infty}^\infty...\sum_{j_n=-\infty}^\infty g(j_1,j_2,...,j_n)f(n_1-j_1,n_2-j_2,...,n_n-j_n),
\end{align*}
where $f$ is our input signal (or image or video) and $g$ the kernel; it is a ``filter" of the input image. %g has length m.
%Here we define convolution in 1D as follows:\\
%$(f * g)(i) = \sum_{j=1}^m g(j).f(i-j+m/2)$, 
If the filter is of a size smaller than the image, then  we can assume that the value of the input image and kernel is 0 everywhere outside the boundary.
The kernel can just be the identity kernel which is a single pixel with a value of 1, or in 2D, the kernel can be an averaging or mean filter. For instance the mean kernel for a $3\times 3$ image is a $3\times 3$ image where every entry is $\frac{1}{9}$.

For the numerical results, 
we are using a Gaussian filter to blur our 3D data. 
The PSF is a  3D gaussian function is defined as 
$$\frac{1}{\sigma^3(2\pi)^{\frac{3}{2}}}\exp\left(-\frac{x^2+y^2+z^2}{2\sigma^2}\right).$$
To blur an image and to recover a high quality deblurred image, we must consider its boundary which is the scene outside the boundaries of an image. Ignoring the boundary can give us a reconstruction that has some artifacts at the edges. %That is why we must impose a boundary condition before working on an image by making some assumptions about the boundary. 
There are different types boundary conditions we can use; see \cite[Ch.~4]{DeblurBook} for more details. %The simplest one is to assume the image is black (i.e. zero) outside the boundary. That is the same as ignoring the outside border of an image. This condition condition would be great if our exact image is mostly zero at the boundary which is rarely the case in real life application. If our image is actually nonzero on the outside the border and we are using this condition our reconstruction can be very bad at the boundary.

We will focus on the periodic boundary condition, which we extend to the N-dimensional case. For periodic boundary conditions in 2D, 
the PSF  is a block circulant with circulant blocks (BCCB).
Note that if the PSF is a BCCB matrix then it is normal and its unitary spectral decomposition is ${A}={F}^*\Lambda {F}$ where ${F}$ is a 2D unitary discrete Fourier Transform.
We can also use this decomposition to find the eigenvalues since a circulant matrix is only defined by its first row or column; see \cite{DeblurBook}.
Before going to the periodic boundary condition for tensors, we need to define the structure of a tensor and then show how we can extend the idea of circulant block to the tensor case.

A tensor $\mathcal{A}\in \mathbb{R}^{I\times J\times K}$ is called totally diagonal if $a_{ijk}$ is nonzero only if $i=j=k$. Moreover, a third order tensor $\mathcal{A}\in \mathbb{R}^{I\times J\times K}$ is called partially diagonal; i.e., diagonal in two modes or more modes for every $k$, $\mathcal{A}(i,j,k)=0,$ for $i\neq j$.

% Let $1<t\leq N$ be a natural number, and let $\{s_1,..., s_t\}$ be a subset of \{1,....N\}. A tensor $\mathcal{A}\in \mathbb{R}^{I_1\times ...\times I_N}$ is called $\{s_1,..., s_t\}-diagonal$, if  $a_{i_1...i_N}$ can be  nonzero only if $i_{s_1}=i_{s_2}=...=i_{s_t}$.
%\end{defn}

%\begin{defn}
 Considering a third-order tensor $\mathcal{A}\in \mathbb{R}^{n\times n\times n}$, where for every $k$, the $\mathcal{A}(:,:,k)$ slices are circulant, i.e., $\mathcal{A}(i,j,k) = \mathcal{A}(i',j',k)$ if $i-j\equiv i'-j'$ (mod $n$).
Hence, $\mathcal{A}$ is circulant with respect to the first and second modes, and we then define $\mathcal{A}$ to be $\{1,2\}$-circulant.

%\end{defn}

%\begin{defn}\label{circ}

Furthermore, a tensor $\mathcal{A}\in \mathbb{R}^{I_1\times ...\times I_n}$ is called $\{l,k\}$-circulant, if $I_l=I_k=n$, and $$\mathcal{A}(:,...,:,i_l,...,i_k,:,...,:)=\mathcal{A}(:,...,:,i'_l,...,i'_k,:,...,:),$$ where $i_l-i_k=i'_l-i'_k$ (mod $n$).
%\end{defn} 
Now using \cref{shift} we have that every column of the matrix $\mathcal{A}(:,:,k)$ can be constructed from $\mathcal{A}(:,1,k)$. This means that for every $j=1,...,n,$ 
% $\mathcal{A}(:,j,k)=C^{j-1}\mathcal{A}(:,1,k)$, so the corresponding also holds for the slices, \\
% $\mathcal{A}(:,j,:)=(C^{j-1})_1.\mathcal{A}(:,1,:)$. 
%More generally,for $\mathcal{A}\in \mathbb{R}^{I_1\times ...\times I_n}$ an \{l,k\}-circulant tensor, we have
\begin{align}\label{circulant}
%\text{ for j= 1,...,n,  } 
\mathcal{A}(:,j,k)&=C^{j-1}\mathcal{A}(:,1,k)\quad \text{ and  }\\
\mathcal{A}(:,j,:)&=(C^{j-1})_1.\mathcal{A}(:,1,:),
%More generally,for $\mathcal{A}\in \mathbb{R}^{I_    
\end{align} where $1$ is in the $k$th mode of $\mathcal{A}$ in \cref{circulant}  and the shift matrix is defined as
\begin{equation}\label{cirC}
C = \begin{bmatrix} 
    0 &         &        & 1 \\
    1 &  \ddots &        & \\
      & \ddots  & \ddots & \\
    0 &         & 1      & 0 
    \end{bmatrix}.
\end{equation}
%The element-wise multiplication of two tensors $\mathcal{X} \in \mathbb{R}^{I\times J\times K}$ and $\mathcal{Y} \in \mathbb{R}^{I\times J\times K}$ is:
%\begin{equation*}
% \mathcal{X}\cdot*\mathcal{Y} = x_{ijk}y_{ijk}.
%\end{equation*}
%The element-wise division of two tensors $\mathcal{X} \in \mathbb{R}^{I\times J\times K}$ and $\mathcal{Y} \in \mathbb{R}^{I\times J\times K}$ is:
%\begin{equation*}
 %  \mathcal{X} \cdot / \mathcal{Y} = \frac{x_{ijk}}{y_{ijk}}.
%\end{equation*}
%\end{defn}

\subsection{Diagonalization of tensors with circulant structure}

%Note for myself: $a\equiv b$ (mod n) means a-b=kn another way is that a and b have the same remainder when divided by n.
Here we would like to cover the case of 
periodic boundary condition for 
multidimensional matrices. For simplicity,
we  focus on the 3D case with an straightforward
generalization to the N-dimensional case. 
 We use periodic boundary condition because it is one of the most common boundary condition, and it has the advantage that its linear system has a circulant structure, which makes it possible to solve the problem using Fast Fourier Transform (FFT). To achieve this goal we will first define a tensor with circulant structure. %First let's talk about a circulant matrix then we will see how it can be extended to tensors.

%\begin{defn}

%A matrix A is said to be circulant if $a_{ij} = a_{i'j'}$ where $i-j\equiv i'-j'$ (modn). The matrix has the following form,

%\[
%A = \begin{bmatrix} 
    %a_0    & a_{n-1} & \dots  & a_1 \\
    %a_1    & a_0     & \ddots & \vdots\\
    %\vdots & \ddots  & \ddots & a_{n-1}\\
    %a_{n-1}&  \dots  & a_1    & a_0 
    %\end{bmatrix}
%\]

%\end{defn}
We know that a circulant matrix is entirely determined by its first column (or row); i.e., for a matrix $A$ we have% Consider a and $b^T$ be the first column and row of A respectively, and the shift matrix to be
\begin{equation}\label{shift}
%\begin{align*}
A(:,j)=C^{j-1}a , \quad j=1,....,n \qquad \text{ and } \qquad 
 A(i,:)=(C^{i-1}b)^T, \quad i=1,....,n,
%\end{align*}
\end{equation} where $a$ and $b^T$ are the first column and row of $A$ respectively and the shift matrix $C$ is defined as in (\ref{cirC}).
%\begin{equation*}
 %   C =\left[ \begin{array}{cccc}
 %  0 &         &        & 1 \\
  %  1 &  \ddots &        & \\
  %   & \ddots  & \ddots & \\
  %  0 &         & 1      & 0   \end{array}\right]
%\end{equation*} 
Hence, any circulant matrix can be diagonalized by a Fourier matrix \cite{Davis}.
Note that the fact that the columns (or rows) of a circulant matrix can be written in terms of the power of the shift matrix $C$ times the first column (row), and this is what allows it to be diagonalized using the discrete Fourier transform.

% \begin{eg} $\mathcal{A}\in \mathbb{R}^{4\times 4\times 3}$ is \{1,2\}-circulant
% where $\mathcal{A}(:,:,1)=\begin{bmatrix}
% 1 & 2 & 3 & 4\\
% 4 & 1 & 2 & 3\\
% 3 & 4 & 1 & 2\\
% 2 & 3 & 4 & 1
% \end{bmatrix}$, $\mathcal{A}(:,:,2)=\begin{bmatrix}
% 13 & 5  & 17 & 18\\
% 18 & 13 & 5  & 17\\
% 17 & 18 & 13 & 5\\
% 5  & 17 & 18 & 13
% \end{bmatrix}$, $\mathcal{A}(:,:,1)=\begin{bmatrix}
% 0  & 9  & 30 & 11\\
% 11 & 0  & 9  & 30\\
% 30 & 11 & 0  & 9\\
% 9  & 30 & 11 & 0
% \end{bmatrix}$ \\

% Therefore for every k=1,2,3, $\mathcal{A}(:,:,k)$ is circulant and by
% (\ref{circulant})\\ 

% $\mathcal{A}(:,2,:)=(C)_1.\mathcal{A}(:,1,:)=\begin{bmatrix}
% 0 & 0 & 0 & 1\\
% 1 & 0 & 0 & 0\\
% 0 & 1 & 0 & 0\\
% 0 & 0 & 1 & 0
% \end{bmatrix}\begin{bmatrix}
% 1 & 13 & 0\\
% 4 & 18 & 11\\
% 3 & 17 & 30\\
% 2  & 5 & 9
% \end{bmatrix}=\begin{bmatrix}
% 2 & 5 & 9\\
% 1 & 13 & 0\\
% 4 & 18 & 11\\
% 3 & 17 & 30
% \end{bmatrix}$, \\

% $\mathcal{A}(:,2,:)$ is a cyclic shift of $\mathcal{A}(1,:,:)$ in the mode-1.\\

% $\mathcal{A}(2,:,:)=(C)_2.\mathcal{A}(1,:,:)=\begin{bmatrix}
% 0 & 0 & 0 & 1\\
% 1 & 0 & 0 & 0\\
% 0 & 1 & 0 & 0\\
% 0 & 0 & 1 & 0
% \end{bmatrix}\begin{bmatrix}
% 1 & 13 & 0\\
% 2 & 5 & 9\\
% 3 & 17 & 30\\
% 4  & 18 & 11
% \end{bmatrix}=\begin{bmatrix}
% 4 & 18 & 11\\
% 1 & 13 & 0\\
% 2 & 5 & 9\\
% 3 & 17 & 30
% \end{bmatrix}$, \\

% $\mathcal{A}(2,:,:)$ is then a cyclic shift of $\mathcal{A}(1,:,:)$ in the mode-2.
% \end{eg}
%Next let's talk about diagonilization of an \{l,k\}-circulant tensor.

\begin{prop}[Diagonalization of a circulant matrix \cite{Elden}]
Let $A\in \mathbb{R}^{n\times n}$ be a circulant matrix. Then A is diagonalized by a Fourier matrix $F$ as 
\begin{equation}\label{propcirc}
\begin{split}
A=F^*\Lambda_1F, \quad  \Lambda_1=\diag(\sqrt{n}Fa),\\
A=F\Lambda_2F^*, \quad  \Lambda_2=\diag(\sqrt{n}Fb), 
\end{split}
\end{equation}
where a and $b^T$ are the first column and row of A, respectively, and $\Lambda_1$ and $\Lambda_2$ are conjugate (i.e. $\Lambda_1 =  \Bar{\Lambda}_2$).
\end{prop} 
The same idea can be used here for tensors as well. If we consider tensors whose slices are circulant with respect to a pair of modes, then we can write them in terms of powers of the shift matrix $C$, which in turn makes it possible to diagonalize the tensor using the Fourier transform. Considering shifts of slices, it is straightforward to obtain the following relations, %the general version of the equation in \cref{circ}.

\begin{lemma}[see \cite{Elden}]
If $\mathcal{A} \in \mathbb{R}^{I_1 \times \cdot \times I_N}$ is $\{l,k\}$-circulant, then for every $1 \leq i_k \leq I_k$ we have
\[\mathcal{A}(:,\cdots,:,\cdots,i_k,\cdots,:)=(C^{i_k -1})_l \cdot \mathcal{A}(:,\cdots,:,\cdots,1,\cdots,:), \]
and for every $1 \leq i_l \leq I_l$
\[\mathcal{A}(:,\cdots,i_l,\cdots,:,\cdots,:)=(C^{i_l -1})_k \cdot \mathcal{A}(:,\cdots,1,\cdots,:,\cdots,:), \]
where $1$ is in the $k$th and $l$th mode of $\mathcal{A}$ in the first and second equations, respectively.
\end{lemma}

\begin{thm} \label{circco}
Let $\mathcal{A}\in \mathbb{R}^{I_1\times...\times I_6}$ be such that for every $i=1,2,3,$ $\mathcal{A}$ is $\{i,i+3\}$-circulant. Then for every $i_{4},...,i_{2n}$, $$\mathcal{A}(:,:,:,i_{4},...,i_{6})=(C^{i_{4}-1},...,C^{i_{6}-1})_{1,2,3}\mathcal{A}(:,:,:,1,...,1).$$
\end{thm}

%The proof of \cref{circco} in \cite{Elden}.

\begin{thm}[see \cite{Elden}]\label{2ndthm}
Let $\mathcal{A}\in \mathbb{R}^{I_1\times I_2\times I_3}$be $\{l,k\}$-circulant. Then $\mathcal{A}$ satisfies $\mathcal{A}=(F^*,F)_{l,k}.\Omega$, where $\Omega$ is a $\{l,k\}$-diagonal tensor with diagonal elements $\mathcal{D}=(\sqrt{n}F)_l.\mathcal{A}(:,...,1,:,...,:)$; here $1$ is in the $k$th mode of $\mathcal{A}$. In particular, $\Omega(\Bar{i})=\delta_{i_li_k}\mathcal{D}(\Bar{i_k})$, with the multi-indices $\Bar{i}=(i_1,i_2,i_3)$ and $\Bar{i_k}=(i_1,...,i_{k-1},i_{k+1},...,i_3)$.
\end{thm}

% \begin{proof}
% Let's assume l=1 and k=2, then by \cref{propcirc}, then for every fixed $i_3,...,i_N$ we have \\
% $(F,F^*)_{1,2}\cdot \mathcal{A}(:,:,i_3,...,i_N)$.\\
% Let's define \\
% $\mathcal{D} = (\sqrt{n}F)_1\cdot\mathcal{A}(:,1,:,...,:)$ \\
% and $\Omega(:,:,i_3,...,i_N) = (F.F^*)_{1,2}\cdot\mathcal{A}(:,:,i_3,...,i_N)$,\\ where for a tensor $B\in \mathbb{R}^{I\times J\times K}$, $B=(\mathcal{W})_1\cdot\mathcal{A}$ is defined as mode-1 product of a tensor $\mathcal{W}$ by $\mathcal{A}$; i.e. $b_{ijk} = \sum_{m=1}^M w_{im}a_{mjk}$.

% Thus, $\Omega = (F,F^*)_{1,2}\cdot\mathcal{A}$ is \{1,2\}-diagonal, and its diagonal elements are $\mathcal{D}$, i.e., $\Omega \Bar{(i)} = \delta_{i_1,i_2}\mathcal{D}\Bar{(i_2)}$
% \end{proof}

\begin{thm}[\cite{Elden}]\label{corollary 2}
 Let $\mathcal{A}\in \mathbb{R}^{I_1\times ...\times I_{6}}$ be such that for every $i=1,2,3$, $\mathcal{A}$ is $\{i,i+3\}$-circulant, and $\mathcal{X}\in \mathbb{R}^{I_1\times I_2\times I_3}$. The linear system of equations \begin{equation}\label{system}
 \mathcal{Y}=\langle\mathcal{A},\mathcal{X}\rangle_{1:3;1:3},\end{equation} is equivalent to $\Bar{\mathcal{Y}}=\mathcal{D}.*\Bar{\mathcal{X}}$, where $\langle\mathcal{A},\mathcal{X}\rangle_{1:3;1:3}$ are the contracted product of two tensors (see \cite{prod} chapter 2),
 \begin{align*} \Bar{\mathcal{Y}}&=(F^*,...,F^*)_{1:N}.\mathcal{Y},
 \qquad \Bar{\mathcal{X}}=(F^*,...,F^*)_{1:N}.\mathcal{X} \quad \text { and } \\
 \mathcal{D}&=(\sqrt{I_1}F,...,\sqrt{I_N}F)_{1:N}.\mathcal{A}(:,...,:,1,...,1).
 \end{align*} 
\end{thm}

% \begin{proof}

% Let $\mathcal{A}\in \mathbb{R}^{I_1\times ...\times I_{2N}}$ and $\mathcal{A}$ is \{i,i+N\}-circulant, for every i=1,...,N then by \cref{2ndthm} we have that\\
% $\mathcal{A} = (F^*,...,F^*,F,...,F)_{1,...,N,N+1,...2N}\cdot\Omega$ where $\Omega$ is a \{1,N+1\},...,\{N,2N\}-diagonal tensor, with diagonal elements $\mathcal{D} = (\sqrt{I_1}F,...,\sqrt{I_n}F)_{1,...,N}\cdot \mathcal{A}(:,...,:,1,1,...,1)$.

% Applying that to \cref{system} we get,
% \begin{equation*}
%     \begin{split}
%         \mathcal{Y} &= \langle(F^*,...,F^*,F,...,F)_{1:2N}\cdot\Omega,\mathcal{X}\rangle_{1:N;1:N}\\
%         &=(F,...,F)_{1:N}\cdot (\langle\Omega,(F^*,...,F^*)_{1:N}\cdot \mathcal{X}\rangle_{1:N;1:N}),
%     \end{split}
% \end{equation*}

% Multiplying the result from mode 1 to N by $F^*$ we get:\\
% $(F^*,...,F^*)_{1:N}\cdot \mathcal{Y} = \langle\Omega,(F^*,...,F^*)_{1:N}\cdot \mathcal{X}\rangle_{1:N;1:N}$.

% Thus, by letting $\Bar{\mathcal{Y}} =(F^*,...,F^*)_{1:N}\cdot \mathcal{Y}$ and $\Bar{\mathcal{X}}= (F^*,...,F^*)_{1:N}\cdot \mathcal{X}$ we have\\
% $\Bar{\mathcal{Y}}=\langle\Omega,\Bar{\mathcal{X}}\rangle_{1:N,1:N}$ which is equivalent to $\Bar{\mathcal{Y}} = \mathcal{D}\cdot*\Bar{\mathcal{X}}$ since $\Omega$ is \{i,i+N\}-diagonal for every i=1,...,N.

% \end{proof}
%In our case  %let's consider the 3-D case. 
From the above theorems, using the periodic boundary condition, we know that the 3D case is just a generalisation of the 1D and 2D cases; see \cite{DeblurBook}. 3D or higher-dimensional cases are handled by increasing the number of modes. 

If we have two third-order tensors $\mathcal{X}\in \mathbb{R}^{n\times n\times n}$ and $\mathcal{Y}\in \mathbb{R}^{n\times n\times n}$ being the true and blurred image, respectively, and $\mathcal{G}$ be the PSF array with center at $\mathcal{G}(l_1,l_2,l_3)$, then 
the rotated PSF will be defined by $\hat{G} = (C^{l_1}J, C^{l_2}J,C^{l_3}J)_{1:3}.\mathcal{G}$. 
Furthermore, the relation between the true and blurred image can be written as tensor-tensor linear system 
$$ \mathcal{Y}=\langle A, \mathcal{X}\rangle_{1:3;1:3},
\qquad 
\mathcal{A}(:,:,:,i,j,k)=(C^{i-1},C^{j-1}, C^{k-1})_{1:3}.\hat{G},$$
where $\mathcal{A}$ is a $\{1,4\}$, $\{2,5\}$ and $\{3,6\}$-circulant tensor.
Hence by \cref{corollary 2} this linear system is equivalent to $\Bar{\mathcal{Y}}=\Bar{\mathcal{G}}.*\Bar{\mathcal{X}}$, 
where 
\begin{align*} 
\Bar{\mathcal{Y}}&= (F^*,F^*,F^*)_{1:3}\cdot\mathcal{Y}, \qquad  \Bar{\mathcal{X}}= (F^*,F^*,F^*)_{1:3}\cdot\mathcal{X} \quad \text{ and } \\
\Bar{\mathcal{G}}&= (\sqrt{n}F,\sqrt{n}F,F)\sqrt{n}_{1:3}\cdot\mathcal{A}(:,:,:,1,1,1) = (\sqrt{n}F,\sqrt{n}F,F)\sqrt{n}_{1:3}\cdot \mathcal{P}.
\end{align*} 
So,
$\mathcal{Y} = \text{\text{fftn}}(\text{fftn}(\hat{G})\cdot*\text{ifft}(\mathcal{X}))$ with the discrete Fourier transform denoted as 
$\text{fftn}$ and  $\text{ifft}$ its inverse. 
Therefore, a naive solution to the deblurring problem can be calculated by
$\mathcal{X} = \text{fftn}(\text{ifft}(\mathcal{Y}\cdot/\text{fftn}(\hat{\mathcal{G}}))$. This, however, is not feasible since $\hat{\mathcal{G}}$
may have very small (or zero) values and thus cannot be inverted.

\section{Numerical Experiments}

All of our experiments were run using MATLAB R2016B version.
Our first experiment we implemented our algorithm on a $240\times240\times114$%3$  
video that is affected by  a $15\times 15\times 3$ Gaussian blur. Also we added some random noise. The PNSR of the blurred and noisy video is 33.978db. The regularization parameter was chosen to be $\lambda=0.01$. The recovered video is given with different PSNR 37.042 Db. Our experiment is illustrated in \cref{fig:video}.  The stopping criteria for this case are assumed to be $m_{\text{max}} = 100$.

\begin{figure}
\centering
    \begin{subfigure}{2.0\linewidth}
    %\centering
    %\includegraphics[width=0.5\textwidth]{deblur_vid1.png}
    \includegraphics[width=0.5\textwidth]{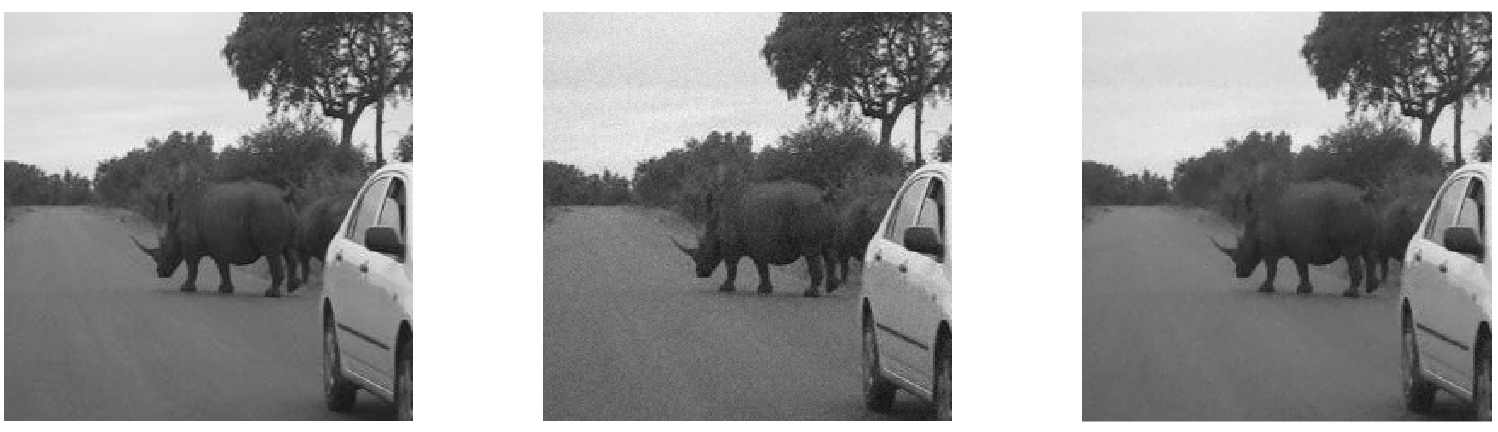}
    %\caption{Image}\label{fig:image1}
    \end{subfigure} %
     \hfill
     \begin{subfigure}{2.0\linewidth}
     %\centering
     %\includegraphics[width=0.5\textwidth]{deblur_vid2.png}
     \includegraphics[width=0.5\textwidth]{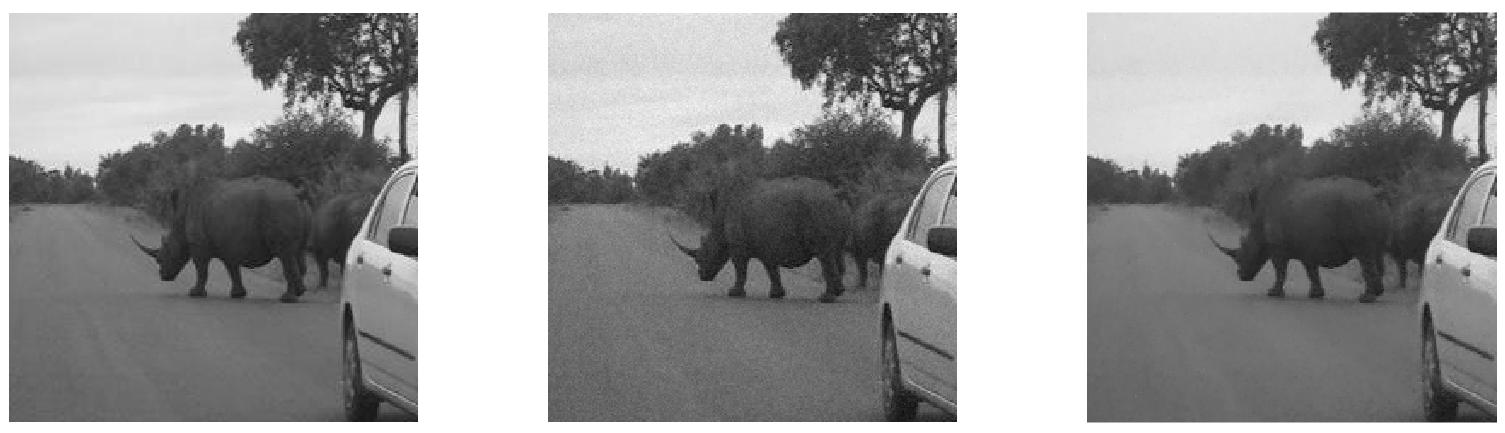}
     %\caption{Image}\label{fig:image12}
   \end{subfigure}

   \hfill
\begin{subfigure}{2.0\linewidth}
    %\centering
   %\includegraphics[width=0.5\textwidth]{deblur_vid3.png}
   \includegraphics[width=0.5\textwidth]{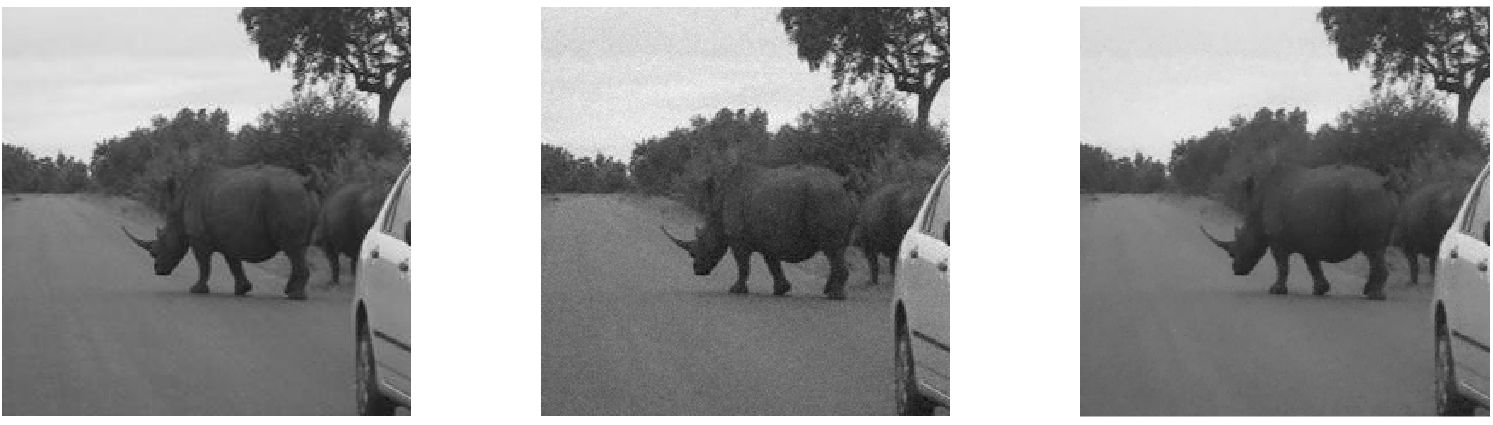}
    %\caption{Image}\label{fig:image12}
   \end{subfigure}

%\RawCaption{\caption{General caption}
\caption{Three frames of the video with True image (left), Noisy input (PNSR 33.978) and recovered image (right) (PNSR 37.042) at each row.}
\label{fig:video}
\end{figure}

Our second experiment we implemented our algorithm on a $200\times200\times3$ colored image from the  MATLAB image processing toolbox. We applied a $15\times 15\times 3$ Gaussian blur and added some random noise. 
Note that in the experiments for \cref{fig:pepperblur}, \cref{fig:pepblur}, and \cref{fig:brickblur}, a blurring is also applied in the color 
mode, which leads to a less colorful input image with ''averaged`` colors.  
The PNSR of the blurred and noisy image is 19.92 Db. The regularization parameter was chosen to be $\lambda=0.02$. The recovered image is given with different PSNR 22.13 Db. Our experiment is illustrated in \cref{fig:pepperblur}.  The stopping criteria for this case is   $m_{\text{max}} = 100$.

\begin{figure}
     \centering
     \begin{subfigure}[b]{0.3\textwidth}
         \centering
         \includegraphics[width=\textwidth]{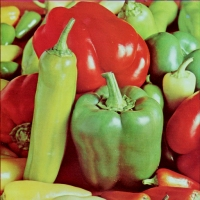}
         \caption{True Image}
        % \label{fig:y equals x}
     \end{subfigure}
     \hfill
     \begin{subfigure}[b]{0.3\textwidth}
         \centering
         \includegraphics[width=\textwidth]{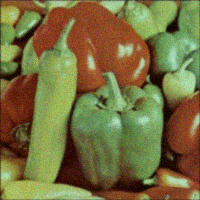}
         \caption{Input (PSNR 19.92)}
         %\label{fig:three sin x}
     \end{subfigure}
     \hfill
     \begin{subfigure}[b]{0.3\textwidth}
         \centering
         \includegraphics[width=\textwidth]{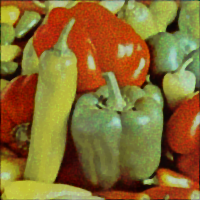}
         \caption{Output (PSNR 22.13)}
         %\label{fig:five over x}
     \end{subfigure}
        \caption{TV-regularized deblurring with PSNR values.}
        \label{fig:pepperblur}
\end{figure}

Our next experiment we implemented our algorithm on a $768\times768\times3$ colored image from the MATLAB image processing toolbox. We applied a $45\times 45\times 3$ Gaussian blur, again with added  random noise. 
The PNSR of the blurred and noisy image is 20.67 Db. The regularization parameter was chosen to be $\lambda=0.01$. The recovered image is given with different PSNR 22.7 Db. Our experiment is illustrated in \cref{fig:pepblur}.  The stopping criteria for this case is  $m_{\text{max}} = 150$.

%\begin{table}[h]
%\centering
%\begin{tabular}{|l | l | l|}
%Method & PSNR (Db) & CPU Times (sec)\\
%\hline
%MFISTA & 24.11 & 7.22\\
%Split Bergman & &
%\end{tabular}
%\caption{CPU AND PSNR VALUES FOR MFISTA and Split Bergman}
%\label{tab:abc}
%\end{table}

\begin{figure}
     \centering
     \begin{subfigure}[b]{0.3\textwidth}
         \centering
         \includegraphics[width=\textwidth]{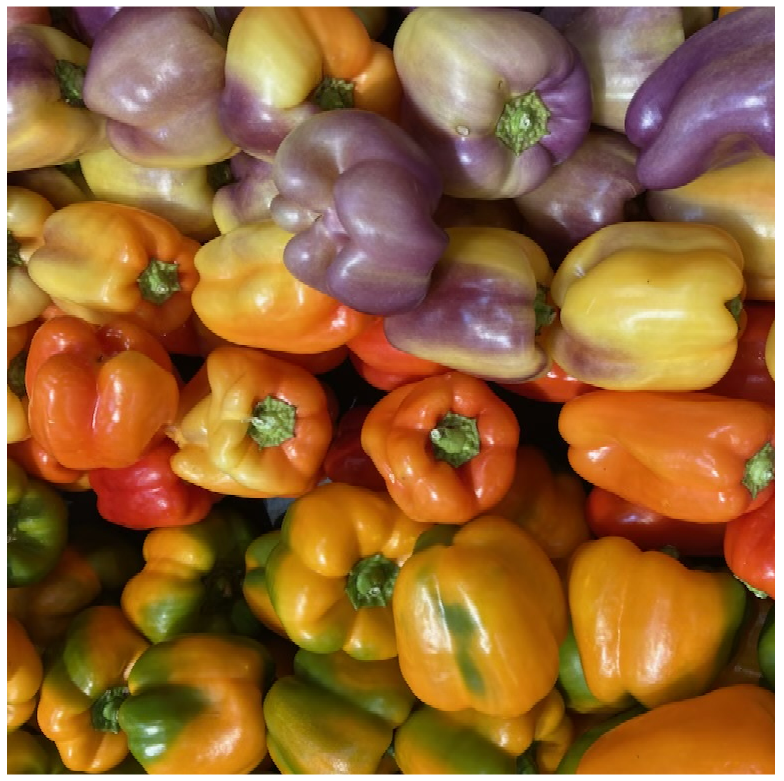}
         \caption{True Image}
        % \label{fig:y equals x}
     \end{subfigure}
     \hfill
     \begin{subfigure}[b]{0.3\textwidth}
         \centering
         \includegraphics[width=\textwidth]{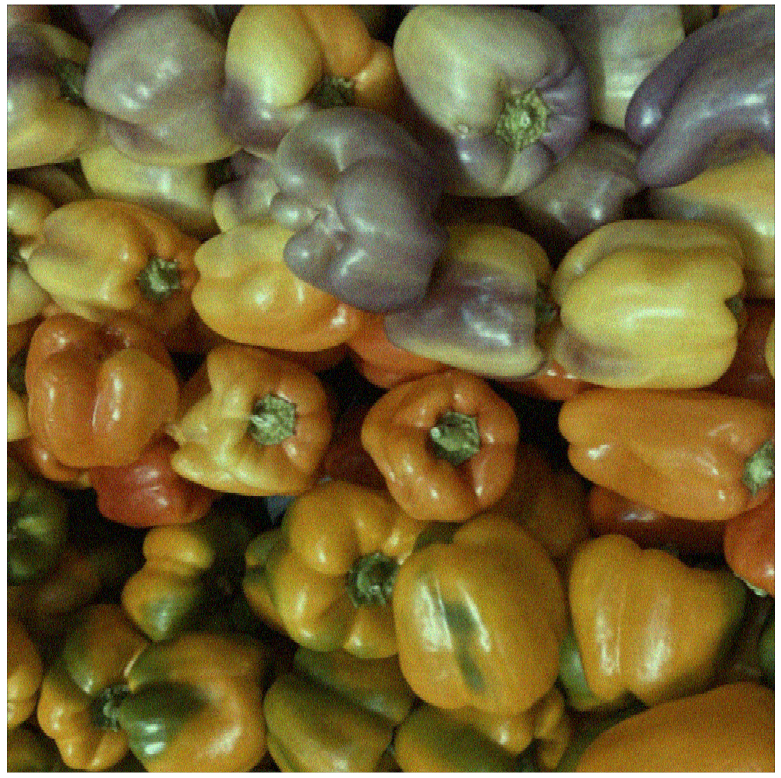}
         \caption{Input (PSNR 20.67)}
         %\label{fig:three sin x}
     \end{subfigure}
     \hfill
     \begin{subfigure}[b]{0.3\textwidth}
         \centering
         \includegraphics[width=\textwidth]{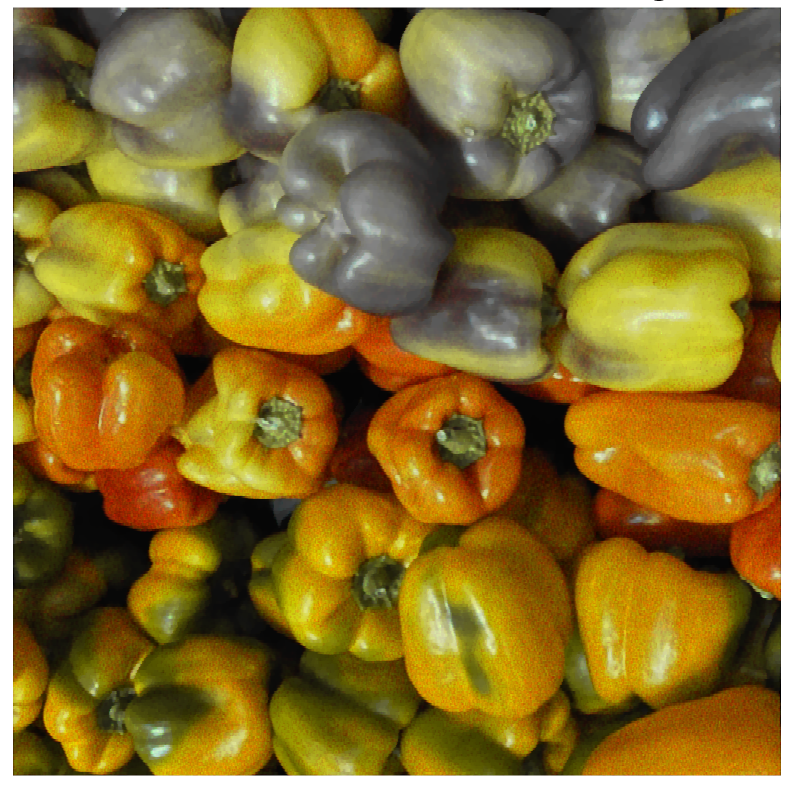}
         \caption{Output (PSNR 22.7)}
         %\label{fig:five over x}
     \end{subfigure}
        \caption{TV-regularized deblurring with PSNR values.}
        \label{fig:pepblur}
\end{figure}

Our last experiment we implemented our algorithm on a $246\times246\times3$ colored image. We applied a $15\times 15\times 3$ Gaussian blur with some random noise. The PNSR of the blurred and noisy image is 19.03 Db. The regularization parameter was chosen to be $\lambda=0.02$. The recovered image is given with different PSNR 22.47 Db. Our experiment is illustrated in \cref{fig:brickblur}.  The stopping criteria for this case is  $m_{\text{max}} = 150$.

\begin{figure}
     \centering
     \begin{subfigure}[b]{0.3\textwidth}
         \centering
         \includegraphics[width=\textwidth]{brick_orig.png}
         \caption{True Image}
        % \label{fig:y equals x}
     \end{subfigure}
     \hfill
     \begin{subfigure}[b]{0.3\textwidth}
         \centering
         \includegraphics[width=\textwidth]{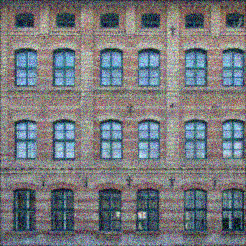}
         \caption{Input (PSNR 19.03)}
         %\label{fig:three sin x}
     \end{subfigure}
     \hfill
     \begin{subfigure}[b]{0.3\textwidth}
         \centering
         \includegraphics[width=\textwidth]{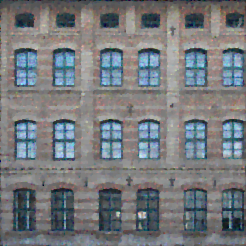}
         \caption{Output (PSNR 22.47)}
         %\label{fig:five over x}
     \end{subfigure}
        \caption{TV-regularized deblurring with PSNR values.}
        \label{fig:brickblur}
\end{figure}

\section{Conclusion}
In this article we applied fast   gradient-based schemes for the constrained total variation based tensor denoising and deblurring problems. 
The methodology is general enough to cover other types of nonsmooth regularizers. The proposed model is simple and relies on combining a dual approach with a fast gradient projection scheme as in \cite{TeboulleBeck}. 
%For the TV-based debluring we have the monotone approach (MFISTA). 
Our  numerical results show a good  performance of the algorithms 
for the considered nonsmooth variational regularization problem for tensors.

\section*{Acknowledgments}
This material is based upon work supported by the National Science Foundation under Grant No. DMS-1439786 while  C. Navasca was in residence at the Institute for Computational and Experimental Research in Mathematics in Providence, RI, during the Model and Dimension Reduction in Uncertain and Dynamic Systems Program. C. Navasca is also in  part supported by National Science Foundation No. MCB-2126374.

\bibliographystyle{siamplain}
%\bibliographystyle{abbrv}
%\bibliography{tensor_reg2}
  
\end{document}